\documentclass{amsart}

\usepackage{xspace, graphicx, epsfig, amssymb, amsmath, amsthm, stmaryrd, color}

\usepackage{mathtools}

\usepackage[multiple]{footmisc}
\usepackage[shortlabels]{enumitem}

\newtheorem{theorem}{Theorem}[section]
\newtheorem{corollary}[theorem]{Corollary}
\newtheorem{lemma}[theorem]{Lemma}

\newtheorem{proposition}[theorem]{Proposition}

\newtheorem{conjecture}[theorem]{Conjecture}
\newtheorem{thm-dfn}[theorem]{Theorem-Definition}
\newtheorem{claim}[theorem]{Claim}
\newtheorem{clm-dfn}[theorem]{Claim-Definition}

\newtheorem{example}[theorem]{Example}
\newtheorem{question}[theorem]{Question}

\theoremstyle{definition}
\newtheorem{definition}[theorem]{Definition}
\newtheorem{remark}[theorem]{Remark}

\numberwithin{equation}{section}

\newcommand{\quash}[1]{}  

\newcommand{\fraka}{{\mathfrak a}}

\newcommand{\frakg}{{\mathfrak g}}

\newcommand{\frakn}{{\mathfrak n}}

\newcommand{\bbC}{{\mathbb C}}

\newcommand{\bbQ}{{\mathbb Q}}
\newcommand{\bbR}{{\mathbb R}}

\newcommand{\bbZ}{{\mathbb Z}}

\newcommand{\calB}{{\mathcal B}}

\newcommand{\calI}{{\mathcal I}}

\newcommand{\calM}{{\mathcal M}}

\newcommand{\calO}{{\mathcal O}}

\newcommand{\calU}{{\mathcal U}}

\newcommand{\ba}{{\boldsymbol{a}}}

\newcommand{\vol}{\textnormal{vol}}

\newcommand{\mtrx}[4]{\left( \begin{array}{cc} #1 & #2 \\ #3 & #4 \end{array} \right)}

\newcommand{\Hom}{{\textnormal{Hom}}}

\newcommand{\Ms}[3]{{M_{ #1 , #2 } ( #3 )}}
\newcommand{\Mss}[3]{{M_{ #1 , #2 } ( #3 )}}

\newcommand{\ms}[3]{{M^{\circ}_{ #1 , #2 } ( #3 )}}

\begin{document}

\title{On tempered representations}
\author[David Kazhdan and Alexander Yom Din]{David Kazhdan and Alexander Yom Din}

\begin{abstract}
	Let $G$ be a unimodular locally compact group. We define a property of irreducible unitary $G$-representations $V$ which we call c-temperedness, and which for the trivial $V$ boils down to F{\o}lner's condition (equivalent to the trivial $V$ being tempered, i.e. to $G$ being amenable). The property of c-temperedness is a-priori stronger than the property of temperedness.
	
	\medskip
	
	We conjecture that for semisimple groups over local fields temperedness implies c-temperedness. We check the conjecture for a special class of tempered $V$'s, as well as for all tempered $V$'s in the cases of $G := SL_2 (\bbR)$ and of $G = PGL_2 (\Omega)$ for a non-Archimedean local field $\Omega$ of characteristic $0$ and residual characteristic not $2$. We also establish a weaker form of the conjecture, involving only $K$-finite vectors.
	
	\medskip
	
	In the non-Archimedean case, we give a formula expressing the character of a tempered $V$ as an appropriately-weighted conjugation-average of a matrix coefficient of $V$, generalizing a formula of Harish-Chandra from the case when $V$ is square-integrable.
\end{abstract}

\maketitle

\setcounter{tocdepth}{1}
\tableofcontents


\section{Introduction}\label{sec intro}

\subsection{}

Throughout the paper, we work with a unimodular second countable locally compact group $G$, and fix a Haar measure $dg$ on it. In the introduction, in \S\ref{ssec intro 1} - \S\ref{ssec intro 2.5} $G$ is assumed semisimple over a local field, while in \S\ref{ssec intro 3} - \S\ref{ssec intro 4} there is no such assumption. After the introduction, in \S\ref{sec Kfin} - \S\ref{sec proof of SL2R} $G$ is assumed semisimple over a local field, while in \S\ref{sec bi tempered} - \S\ref{sec ctemp is temp} there is no such assumption. Unitary representations of $G$ are pairs $(V , \pi)$, but for lightness of notation we denote them by $V$, keeping $\pi$ implicit.

\subsection{}\label{ssec intro 1}

Assume that $G$ is a semisimple group over a local field\footnote{So, for example, $G$ can be taken $SL_n (\bbR)$ or $SL_n (\bbQ_p)$.}. The characterization of temperedness of irreducible unitary $G$-representations in terms of the rate of decrease of $K$-finite matrix coefficients is well-studied (see for example \cite{Wa,CoHaHo,Be}). Briefly, fixing a maximal compact subgroup $K \subset G$, an irreducible unitary $G$-representation $V$ is tempered if and only if for every two $K$-finite vectors $v_1 , v_2 \in V$ there exists $C>0$ such that $$ |\langle gv_1 , v_2 \rangle | \leq C \cdot \Xi_G (g)$$ for all $g \in G$, where $\Xi_G : G \to \bbR_{\ge 0}$ is Harish-Chandra's $\Xi$-function (see \S\ref{ssec V1 and Xi} for a reminder on the definition of $\Xi_G$). When considering matrix coefficients of more general vectors, differentiating between tempered and non-tempered irreducible unitary $G$-representations becomes more problematic, as the following example shows.

\begin{example}[see Claim \ref{clm counterexample}]\label{rem counterexample}
	Let $G := PGL_2 (\Omega)$, $\Omega$ a local field. Denote by $A \subset G$ the subgroup of diagonal matrices. Given a unitary $G$-representation $V$ let us denote $$ \calM_V (A) := \left\{ a \mapsto \langle a v_1 , v_2 \rangle \right\}_{v_1 , v_2 \in V} \subset C(A),$$ i.e. the set of matrix coefficients of $V$ restricted to $A$. Let us also denote $$ \widehat{L^1} (A) :=   \left\{ a \mapsto \int_{\hat{A}} \chi (a) \cdot \phi (\chi) \cdot d \chi \right\}_{\phi \in L^1 (\hat{A})} \subset C(A),$$ i.e. the set of Fourier transforms of $L^1$-functions on $\hat{A}$. Then for any non-trivial irreducible unitary $G$-representation $V$ we have $$ \calM_V (A) = \widehat{L^1} (A).$$
\end{example}

The remedy proposed in this paper is that, instead of analysing the pointwise growth of matrix coefficients, we analyse their ``growth in average", i.e. the behaviour of integrals of norm-squared matrix coefficients over big balls.

\subsection{}

We fix a norm\footnote{In the non-Archimedean case, norms on finite-dimensional vector spaces are discussed, for example, in \cite[Chapter II, \S 1]{We}.} $|| - ||$ on the vector space $\frakg := {\rm Lie} (G)$ and consider also the induced operator norm $|| - ||$ on $\textnormal{End} (\frakg)$. We  define the ``radius" function $\mathbf{r} : G \to \bbR_{\ge 0}$ by $$ \mathbf{r} (g) := \log \left( \max \{ || \textnormal{Ad} (g) ||,  || \textnormal{Ad} (g^{-1}) || \} \right)$$ where $\textnormal{Ad} : G \to \textnormal{Aut} (\frakg)$ is the adjoint representation. We denote then by $G_{<r} \subset G$ the subset of elements $g$ for which $\mathbf{r} ( g) < r$.

\begin{conjecture}[``asymptotic Schur orthogonality relations"]\label{main conj red exact}
	Let $V$ be a tempered irreducible unitary $G$-representation. There exist $\mathbf{d}(V) \in \bbZ_{\ge 0}$ and $\mathbf{f} (V) \in \bbR_{>0}$ such that  for all $v_1 , v_2 , v_3 , v_4 \in V$ we have $$ \lim_{r \to +\infty} \frac{\int_{G_{< r}} \langle g v_1 , v_2 \rangle \overline{\langle gv_3 , v_4 \rangle } \cdot dg}{  r^{\mathbf{d} (V)}} = \frac{1}{\mathbf{f} (V)} \cdot \langle v_1 , v_3 \rangle  \overline{\langle v_2 , v_4 \rangle}.$$
\end{conjecture}

\begin{remark}[see Claim \ref{clm doesnt depend on norm}]\label{rem doesnt depend on norm}
	The validity of Conjecture \ref{main conj red exact}, as well as the resulting invariants $\mathbf{d} (V)$ and $\mathbf{f} (V)$ (and of other similar results/conjectures below - see the formulation of Claim \ref{clm doesnt depend on norm}), do not depend on the choice of the norm $|| - ||$ on $\frakg$ (used to construct the subsets $G_{<r}$).
\end{remark}

\begin{remark}[see Remark \ref{rem ctemp red is temp}]
	An irreducible unitary $G$-representation $V$ for which the condition of Conjecture \ref{main conj red exact} is verified is tempered.
\end{remark}

\begin{remark}
	In the notation of Conjecture \ref{main conj red exact}, $\mathbf{d} (V) = 0$ if and only if $V$ is square-integrable. In that case, $\mathbf{f} (V)$ is the well-known formal degree of $V$.
\end{remark}

\begin{remark}[following from Proposition \ref{prop c-tempered orth rel cross}]
	Let $V$ and $W$ be two tempered irreducible unitary $G$-representations for which Conjecture \ref{main conj red exact} holds, and which are non-isomorphic. Then for all $v_1 , v_2 \in V$ and $w_1 , w_2 \in W$  one has $$ \lim_{r \to +\infty} \frac{\int_{G_{<r}} \langle g v_1 , v_2 \rangle \overline{\langle g w_1 , w_2 \rangle} \cdot dg}{r^{(\mathbf{d} (V) + \mathbf{d} (W))/2}} = 0.$$
\end{remark}

\subsection{}\label{ssec intro 1.25}

We show the following statement, weaker than Conjecture \ref{main conj red exact}:

\begin{theorem}[see \S\ref{sec Kfin}]\label{thm main Kfin}
	Let $V$ be a tempered irreducible unitary $G$-representation and $K \subset G$ a maximal compact subgroup. There exists $\mathbf{d}(V) \in \bbZ_{\ge 0}$ such that:
	\begin{enumerate}
		\item if $G$ is non-Archimedean, there exists $\mathbf{f} (V) \in \bbR_{>0}$ such that for all $K$-finite\footnote{When $G$ is non-Archimedean $K$-finite is the same as smooth (in particular does not depend on $K$).} $v_1 , v_2 , v_3 , v_4 \in V$ we have $$ \lim_{r \to +\infty} \frac{\int_{G_{< r}} \langle g v_1 , v_2 \rangle  \overline{\langle gv_3 , v_4 \rangle } \cdot dg}{ r^{\mathbf{d} (V)}} = \frac{1}{\mathbf{f} (V)} \cdot \langle v_1 , v_3 \rangle  \overline{\langle v_2 , v_4 \rangle}.$$
		\item If $G$ is Archimedean, for any given non-zero $K$-finite vectors $v_1 , v_2 \in V$ there exists $C(v_1 , v_2) >0$ such that $$ \lim_{r \to +\infty} \frac{\int_{G_{< r}} | \langle g v_1 , v_2 \rangle |^2 \cdot dg}{ r^{\mathbf{d} (V)}} = C (v_1 , v_2).$$
	\end{enumerate}
\end{theorem}

\begin{remark}
	We expect that it should not be very difficult to establish the statement of item $(1)$ of Theorem \ref{thm main Kfin} also in the Archimedean case, instead of the weaker statement of item $(2)$.
\end{remark}

Concentrating on the non-Archimedean case for simplicity, Theorem \ref{thm main Kfin} has as a corollary the following proposition, a generalization (from the square-integrable case to the tempered case) of a formula of Harish-Chandra (see \cite[Theorem 9]{Ha2}), expressing the character as a conjugation-average of a matrix coefficient.

\begin{definition}
	Assume that $G$ is non-Archimedean. We denote by $C^{\infty} (G)$ the space of (complex-valued) smooth functions on $G$ and by $D_c^{\infty} (G)$ the space of smooth distributions on $G$ with compact support. We denote by $C^{-\infty} (G)$ the dual to $D_c^{\infty} (G)$, i.e. the space of generalized functions on $G$ (thus we have an embedding $C^{\infty} (G) \subset C^{-\infty} (G)$). Given an admissible unitary $G$-representation $V$, we denote by $\Theta_V \in C^{-\infty} (G)$ the character of $V$.
\end{definition}

\begin{proposition}[see \S\ref{ssec proof of prop formula ch}]\label{prop formula ch}
	Let $V$ be a tempered irreducible unitary $G$-representation. Let $v_1 , v_2 \in V$ be smooth vectors. Denote by $m_{v_1 , v_2} \in C^{\infty}(G) \subset C^{-\infty} (G)$ the matrix coefficient $m_{v_1 , v_2} (g) := \langle g v_1 , v_2 \rangle$. Denoting $({}^g m) (x) := m(g^{-1} x g)$, the limit $$\lim_{r \to +\infty} \frac{\int_{G_{<r}} {}^g m_{v_1 , v_2} \cdot dg}{r^{\mathbf{d} (V)}}$$ exists in $C^{-\infty}(G)$, in the sense of weak convergence of generalized functions (i.e. convergence when paired against every element in $D_c^{\infty} (G)$), and is equal to $$ \frac{\langle v_1 , v_2 \rangle}{\mathbf{f} (V)} \cdot \Theta_V.$$
\end{proposition}

\subsection{}\label{ssec intro 1.5}

We are able to verify Conjecture \ref{main conj red exact} in some cases.

\begin{theorem}[see Theorem \ref{thm V1 c temp}]\label{thm intro slowest}
	Conjecture \ref{main conj red exact} is true for the principal series irreducible unitary representation of ``slowest decrease", i.e. the unitary parabolic induction of the trivial character via a minimal parabolic subgroup.
\end{theorem}

Here is the main result of the paper:

\begin{theorem}[see \S\ref{sec proof of SL2R}]\label{thm intro SL2R}
	Conjecture \ref{main conj red exact} is true for all tempered irreducible unitary representations of $G := SL_2 (\bbR)$ and of $G := PGL_2 (\Omega)$, where $\Omega$ is a non-Archimedean field of characteristic $0$ and of residual characteristic not equal to $2$.
\end{theorem}

\subsection{}\label{ssec intro 2}

The proposition that follows shows that a seemingly weaker property implies that of Conjecture \ref{main conj red exact}.

\begin{definition}
	Given a unitary $G$-representation $V$ and vectors $v_1 , v_2 \in V$ we define $$ \Ms{v_1}{v_2}{r} := \int_{ G_{<r}} |\langle g v_1 , v_2 \rangle |^2 \cdot dg.$$
\end{definition}

\begin{proposition}[see \S\ref{ssec clms red 1}]\label{prop from folner to exact}

	Let $V$ be an irreducible unitary $G$-representation. Let $v_0 \in V$ be a unit vector such that the following holds:
	
	\begin{enumerate}
		\item For any vectors $v_1 , v_2 \in V$ we have $$ \underset{r \to +\infty}{\limsup} \frac{ \Ms{v_1}{v_2}{r} }{ \Ms{v_0}{v_0}{r} } < +\infty.$$
		\item For any vectors $v_1 , v_2 \in V$ and $r^{\prime} > 0$ we have $$ \lim_{r \to +\infty} \frac{ \Ms{v_1}{v_2}{r + r^{\prime}} - \Ms{v_1}{v_2}{r - r^{\prime}} }{ \Ms{v_0}{v_0}{r} } = 0.$$
	\end{enumerate}
	
	Then Conjecture \ref{main conj red exact} holds for $V$.
	
\end{proposition}

\begin{question}
	Does item $(1)$ of Proposition \ref{prop from folner to exact} hold for arbitrary irreducible unitary $G$-representations?
\end{question}

\begin{remark}[see Proposition \ref{clm Prop tempered is c-tempered reductive holds}]\label{rem ctemp red is temp}
	An irreducible unitary $G$-representation for which there exists a unit vector $v_0 \in V$ such that conditions $(1)$ and $(2)$ of Proposition \ref{prop from folner to exact} are satisfied is tempered.
\end{remark}

\subsection{}\label{ssec intro 2.5}

After finishing writing the current paper, we have found previous works \cite{Mi} and \cite{An}. Work \cite{Mi} intends at giving an asymptotic Schur orthogonality relation for tempered irreducible unitary representations, but we could not understand its validity; on the first page the author defines a seminorm $||-||_p^2$ on $C^{\infty} (G)$ by a limit, but this limit clearly does not always exist. Work \cite{An} (which deals with the more general setup of a symmetric space) provides an asymptotic Schur orthogonality relation for $K$-finite vectors in a tempered irreducible unitary $G$-representation, in the case when $G$ is real and under a regularity assumption on the central character. This work also seems to provide an interpretation of what we have denoted as $\mathbf{f} (V)$ in terms of the Plancherel density (but it would be probably good to work this out in more detail).

\subsection{}\label{ssec intro 3}

Let now $G$ be an arbitrary unimodular second countable locally compact group. We formulate a property of irreducible unitary $G$-representations which we call c-temperedness (see Definition \ref{def c-temp}). The property of c-temperedness is, roughly speaking, an abstract version of properties (1) and (2) of Proposition \ref{prop from folner to exact}. Here $G_{< r} \subset G$ are replaced by a sequence $\{ F_n \}_{n \ge 0}$ of subsets of $G$, which we call a F{\o}lner sequence, whose existence is part of the definition (so that we speak of a representation c-tempered with F{\o}lner sequence $\{ F_n \}_{n \ge 0}$), while the condition replacing property (2) of Proposition \ref{prop from folner to exact} generalizes, in some sense, the F{\o}lner condition for a group to be amenable (i.e. for the trivial representation to be tempered).

\medskip

We show in Corollary \ref{cor c-temp are temp} that any c-tempered irreducible unitary $G$-representation is tempered and pose the question:

\begin{question}\label{main question}
	For which groups $G$ every tempered irreducible unitary $G$-representation is c-tempered with some F{\o}lner sequence?
\end{question}

As before, c-tempered irreducible unitary $G$-representations enjoy a variant of asymptotic Schur orthogonality relations (see Proposition \ref{prop c-tempered orth rel}): \begin{equation}\label{eq orth rel} \lim_{n \to +\infty} \frac{\int_{F_n} \langle g v_1 , v_3 \rangle \overline{\langle gv_2 , v_4 \rangle} \cdot dg }{\int_{F_n} |\langle g v_0 , v_0 \rangle |^2 \cdot dg} = \langle v_1 , v_2 \rangle \overline{\langle v_3 , v_4 \rangle}\end{equation} for all $v_1 , v_2 , v_3 , v_4 \in V$ and all unit vectors $v_0 \in V$. Also, we have a variant for a pair of non-isomorphic representations (see Proposition \ref{prop c-tempered orth rel cross}).

\begin{definition}
	Let us say that two irreducible unitary $G$-representations are twins if their closures in $\hat{G}$ (w.r.t. the Fell topology) coincide.
\end{definition}

\begin{question}
	Let $V_1$ and $V_2$ be irreducible unitary $G$-representations and assume that $V_1$ and $V_2$ are twins. Suppose that $V_1$ is c-tempered with F{\o}lner sequence $\{ F_n \}_{n \ge 0}$.
	\begin{enumerate}
		\item Is it true that $V_2$ is also c-tempered with F{\o}lner sequence $\{ F_n \}_{n \ge 0}$?
		\item If so, is it true that for unit vectors $v_1 \in V_1$ and $v_2 \in V_2$ we have $$ \lim_{n \to +\infty} \frac{\int_{F_n} |\langle g v_1 , v_1 \rangle |^2 \cdot dg}{\int_{F_n} |\langle g v_2 , v_2 \rangle|^2 \cdot dg} = 1?$$
	\end{enumerate}
\end{question}

\subsection{}\label{ssec intro 4}

For many groups there exist tempered representations with the slowest rate of decrease of matrix coefficients. For such representations it is often much easier to prove analogs of c-temperedness or of orthogonality relation (\ref{eq orth rel}) than for other representations - as exemplified by Theorem \ref{thm intro slowest} above. See \cite{BoGa} for hyperbolic groups. 

\subsection{} Alexander Yom Din's research was supported by the ISRAEL SCIENCE FOUNDATION (grant No 1071/20). David Kazhdan's research was partially supported by ERC grant No 669655. We would like to thank Pavel Etingof for great help with the proof of Claim \ref{clm SL2 3} in the case $G = SL_2 (\bbR)$, which was present in a prior draft of the paper, before we encountered the work \cite{BrCoNiTa}. We would like to thank Vincent Lafforgue for a very useful discussion. We thank Erez Lapid for useful correspondence.

\subsection{}\label{ssec notation} Throughout the paper, $G$ is a unimodular second countable locally compact group. We fix a Haar measure $dg$ on $G$, as well as Haar measures on the other unimodular groups we encounter ($dk$ on the group $K$, etc.). We denote by $\textnormal{vol}_G (-)$ the volume with respect to $dg$.

\medskip

All unitary $G$-representations are on separable Hilbert spaces. 

\medskip

Given a unitary $G$-representation $V$, vectors $v_1 , v_2 \in V$ and a measurable subset $F \subset G$, we denote $$ \Mss{v_1}{v_2}{F} := \int_{F} |\langle gv_1 , v_2 \rangle |^2 \cdot dg.$$ So in the case of a semisimple group over a local field as above, we have set $$ \Ms{v_1}{v_2}{r} := \Mss{v_1}{v_2}{G_{<r}}.$$

\medskip

We write $L^2 (G) := L^2 (G , dg)$, considered as a unitary $G$-representation via the right regular action.

\medskip

Given Hilbert spaces $V$ and $W$, we denote by $\calB (V; W)$ the space of bounded linear operators from $V$ to $W$, and write $\calB (V) := \calB (V; V)$.

\medskip

We write $F_1 \smallsetminus F_2$ for set differences and $F_1 \triangle F_2 := (F_1\smallsetminus F_2) \cup (F_2 \smallsetminus F_1)$ for symmetric set differences.

\section{Notion of c-temperedness}\label{sec bi tempered}

In this section, let $G$ be a unimodular second countable locally compact group. We introduce the notion of a c-tempered (with a given F{\o}lner sequence) irreducible unitary $G$-representation.

\subsection{}

The following definition aims at a generalization of the hypotheses of Proposition \ref{prop from folner to exact}, so as to make them suitable for a general group.

\begin{definition}\label{def c-temp}
	Let $V$ be an irreducible unitary $G$-representation. Let $F_0 , F_1 , \ldots \subset G$ be a sequence of measurable pre-compact subsets all containing a neighbourhood of $1$. We say that $V$ is \textbf{c-tempered\footnote{``c" stands for ``matrix coefficients".} with F{\o}lner sequence $F_0 , F_1 , \ldots$} if there exists a unit vector $v_0 \in V$ such that the following two conditions are satisfied:
	\begin{enumerate}
		\item For all $v_1 , v_2 \in V$ we have\footnote{The notation $\Mss{-}{-}{-}$ is introduced in \S\ref{ssec notation}.} $$ \limsup_{n \to +\infty} \frac{ \Mss{v_1}{v_2}{F_n} }{ \Mss{v_0}{v_0}{F_n} } < +\infty.$$
		\item For all $v_1 , v_2 \in V$ and all compact subsets $K \subset G$ we have $$ \lim_{n \to +\infty} \frac{\sup_{g_1 , g_2 \in K} \Mss{v_1}{v_2}{F_n \triangle g_2^{-1} F_n g_1} }{ \Mss{v_0}{v_0}{F_n} } = 0.$$
	\end{enumerate}
\end{definition}

\begin{example}
	The trivial unitary $G$-representation is c-tempered with F{\o}lner sequence $F_0 , F_1 , \ldots$ if for any compact $K \subset G$ we have \begin{equation}\label{eq folner} \lim_{n \to +\infty} \sup_{g_1 , g_2 \in K} \frac{\vol_G (F_n \triangle g_2^{-1} F_n g_1)}{\vol_G (F_n)} = 0.\end{equation} By \textbf{F{\o}lner's condition}, the existence of such a sequence is equivalent\footnote{When stating F{\o}lner's condition for the amenability of $G$ it is more usual to consider $g_2^{-1} F_n$ rather than $g_2^{-1} F_n g_1$ in (\ref{eq folner}), i.e. to shift only on one side. However, using, for example, \cite[Theorem 4.1]{Gr} applied to the action of $G \times G$ on $G$, we see that the above stronger ``two-sided" condition also characterizes amenability.} to the trivial irreducible unitary $G$-representation being tempered, i.e. to $G$ being \textbf{amenable}.
\end{example}

\subsection{} Irreducible unitary $G$-representations which are c-tempered satisfy ``asymptotic Schur orthogonality relations":

\begin{proposition}\label{prop c-tempered orth rel}
	Let $V$ be an irreducible unitary $G$-representation. Assume that $V$ is c-tempered with F{\o}lner sequence $F_0 , F_1 , \ldots$ and let $v_0 \in V$ be a unit vector for which the conditions (1) and (2) of Definition \ref{def c-temp} are satisfied. Then for all $v_1 , v_2 , v_3 , v_4 \in V$ we have \begin{equation}\label{eq approx schur bi} \lim_{n \to +\infty} \frac{\int_{F_n} \langle g v_1 , v_2 \rangle \overline{\langle g v_3 , v_4 \rangle} \cdot dg}{ \Mss{v_0}{v_0}{F_n} } = \langle v_1 , v_3 \rangle \overline{\langle v_2 , v_4 \rangle}.\end{equation}
\end{proposition}

\begin{proof}
	
	First, notice that in order to show that the limit in (\ref{eq approx schur bi}) holds, it is enough to show that for every sub-sequence there exists a further sub-sequence of it on which the limit holds. Replacing our sequence by the sub-sequence, it is therefore enough to show simply that there exists a sub-sequence on which the limit holds - which is what we will do.
	
	\medskip
	
	Define bilinear maps\footnote{Recall that $L^2 (G)$ denotes $L^2 (G , dg)$, viewed as a unitary $G$-representation via the right regular action.} $$ S_0 , S_1 , \ldots : V \times \overline{V} \to L^2 (G)$$ by $$ S_n (v_1 , v_2) (g) := \begin{cases} \frac{1}{\sqrt{ \Mss{v_0}{v_0}{F_n} }} \cdot \langle g v_1 , v_2 \rangle, & g \in F_n \\ 0, & g \notin F_n \end{cases}.$$ Clearly those are bounded.
	
	\medskip
	
	\begin{itemize}
		\item The bilinear maps $S_n$ are jointly bounded, i.e. there exists $C>0$ such that $||S_n||^2 \leq C$ for all $n$.
	\end{itemize}
	
	Indeed, by condition (1) of Definition \ref{def c-temp}, for any fixed $v_1 , v_2 \in V$ there exists $C>0$ such that $||S_n (v_1 , v_2)||^2 \leq C$ for all $n$. By the Banach-Steinhaus theorem, there exists $C>0$ such that $||S_n ||^2 \leq C$ for all $n$.
	
	\medskip

	Next, define quadlinear forms $$ \Phi_1 , \Phi_2 , \ldots : V \times \overline{V} \times \overline{V} \times V \to \bbC$$ by $$ \Phi_n (v_1 , v_2 , v_3 , v_4) := \langle S_n (v_1 , v_2) , S_n (v_3 , v_4) \rangle$$
	
	\medskip
	
	\begin{itemize}
		\item The quadlinear forms $\Phi_n$ are jointly bounded, in fact $|| \Phi_n || \leq C$ for all $n$.
	\end{itemize}
	
	This follows immediately from the above finding $|| S_n ||^2 \leq C$ for all $n$.
	
	\medskip

	\begin{itemize}
		\item For all $g_1 , g_2 \in G$ and $v_1 , v_2 , v_3 , v_4 \in V$ we have \begin{equation}\label{eq approx GG-inv} \lim_{n \to +\infty } \left( \Phi_n (g_1  v_1 , g_2 v_2 , g_1 v_3, g_2 v_4 ) - \Phi_n (v_1 , v_2 , v_3 , v_4) \right)= 0.\end{equation}
	\end{itemize}
	
	Indeed, $$ | \Phi_n (g_1 v_1 , g_2 v_2 , g_1 v_3 , v_2 v_4) - \Phi_n (v_1 , v_2 , v_3 , v_4 ) | = $$ $$ =  \frac{ | \int_{F_n} \langle g g_1 v_1 , g_2 v_2 \rangle \overline{\langle g g_1 v_3 , g_2 v_4 \rangle} \cdot dg - \int_{F_n} \langle g v_1 , v_2 \rangle \overline{\langle g v_3 , v_4 \rangle} \cdot dg |}{ \Mss{v_0}{v_0}{F_n} } \leq $$ $$ \leq  \frac{ \int_{F_n \triangle g_2^{-1} F_n g_1 } |\langle g v_1 , v_2 \rangle | \cdot | \langle g v_3 , v_4 \rangle | \cdot dg}{ \Mss{v_0}{v_0}{F_n} } \leq $$ $$ \leq \sqrt{\frac{ \Mss{v_1}{v_2}{F_n \triangle g_2^{-1} F_n g_1} }{ \Mss{v_0}{v_0}{F_n} }} \cdot \sqrt{\frac{ \Mss{v_3}{v_4}{F_n \triangle g_2^{-1} F_n g_1} }{ \Mss{v_0}{v_0}{F_n} }}$$ and the last expression tends to $0$ as $n \to +\infty$ by condition (2) of Definition \ref{def c-temp}.
	
	\medskip
	
	\begin{itemize}
		\item There exists a sub-sequence $0 \leq m_0 <  m_1 <  \ldots$ such that $$ \lim_{n \to +\infty} \Phi_{m_n} (v_1 , v_2 , v_3 , v_4) = \langle v_1 , v_3 \rangle \overline{ \langle v_2 , v_4 \rangle } $$ for all $v_1 , v_2 , v_3 , v_4 \in V$.
	\end{itemize}
	
	By the sequential Banach-Alaouglu theorem (which is applicable since $V$ is separable), we can find a sub-sequence $0 \leq m_0 <  m_1 <  \ldots$ and a bounded quadlinear form $$ \Phi : V \times \overline{V} \times \overline{V} \times V \to \bbC$$ such that $\lim_{n \to +\infty}^{\textnormal{weak-}*} \Phi_{m_n} = \Phi$. Passing to the limit in equation (\ref{eq approx GG-inv}) we obtain that for all $g_1 , g_2 \in G$ and all $v_1 , v_2 , v_3 , v_4 \in V$ we have $$ \Phi (g_1 v_1 , g_2 v_2 , g_1 v_3 , g_2 v_4) = \Phi (v_1 , v_2 , v_3 , v_4).$$ Fixing $v_2 , v_4$, we obtain a bounded bilinear form $\Phi (- , v_2 , - , v_4) : V \times \overline{V} \to \bbC$ which is $G$-invariant, and hence by Schur's lemma is a multiple of the form $\langle - , - \rangle$, i.e. we have a uniquely defined $c_{v_2 , v_4} \in \bbC$ such that $$ \Phi (v_1 , v_2 , v_3 , v_4) = c_{v_2 , v_4} \cdot \langle v_1 , v_3 \rangle$$ for all $v_1 , v_3 \in V$. Similarly, fixing $v_1 , v_3$ we see that we have a uniquely defined $d_{v_1 , v_3} \in \bbC$ such that $$ \Phi (v_1 , v_2 , v_3 , v_4) = d_{v_1 , v_3} \cdot \overline{\langle v_2 , v_4 \rangle}$$ for all $v_2, v_4 \in \bbC$. Since $\Phi (v_0 , v_0, v_0, v_0) = 1$, plugging in $(v_1 , v_2 , v_3 ,v_4) := (v_0 , v_0 , v_0, v_0)$ in the first equality we find $c_{v_0 , v_0} = 1$. Then plugging in $(v_1 , v_2 , v_3 , v_4) := (v_1 , v_0 , v_3 , v_0)$ in both equalities and comparing, we find $d_{v_1 , v_3} = \langle v_1 , v_3 \rangle$. Hence we obtain $$ \Phi (v_1 , v_2 , v_3 , v_4) = \langle v_1 , v_3 \rangle \overline{\langle v_2 , v_4 \rangle} $$ for all $v_1 , v_2 , v_3 , v_4 \in V$.
	
	\medskip
	
	Now, writing explicitly $\Phi_{m_n} (v_1 , v_2 , v_3 , v_4)$, we see that the limit in (\ref{eq approx schur bi}) is valid on our sub-sequence, so we are done, as we explained in the beginning of the proof.
	
\end{proof}

\subsection{} If one unit vector $v_0$ satisfies conditions (1) and (2) of Definition \ref{def c-temp} then all unit vectors do:

\begin{proposition}\label{prop c-tempered all are temp}
	Let $V$ be an irreducible unitary $G$-representation. Assume that $V$ is c-tempered with F{\o}lner sequence $F_0 , F_1 , \ldots$ and let $v_0 \in V$ be a unit vector for which the conditions (1) and (2) of Definition \ref{def c-temp} are satisfied. Then for any unit vector $v'_0 \in V$ the conditions (1) and (2) of Definition \ref{def c-temp} are satisfied.
\end{proposition}

\begin{proof}

	Let $v'_0 \in V$ be a unit vector. From (\ref{eq approx schur bi}) we get $$ \lim_{n \to +\infty} \frac{ \Mss{v'_0}{v'_0}{F_n} }{ \Mss{v_0}{v_0}{F_n} } = 1.$$ This makes the claim clear.
	
\end{proof}

\subsection{}

We also have the following version of ``asymptotic Schur orthogonality relations" for a pair of non-isomorphic irreducible representations:

\begin{proposition}\label{prop c-tempered orth rel cross}
	Let $V$ and $W$ be irreducible unitary $G$-representations. Assume that $V$ and $W$ are c-tempered with the same F{\o}lner sequence $F_0 , F_1 , \ldots$ and let $v_0 \in V$ and $w_0 \in W$ be unit vectors for which the conditions (1) and (2) of Definition \ref{def c-temp} are satisfied. Then for all $v_1 , v_2 \in V$ and $w_1 , w_2 \in W$ we have \begin{equation}\label{eq approx schur bi cross} \lim_{n \to +\infty} \frac{\int_{F_n} \langle g v_1 , v_2 \rangle \overline{\langle g w_1 , w_2 \rangle } \cdot dg}{ \sqrt{\Mss{v_0}{v_0}{F_n}} \sqrt{\Mss{w_0}{w_0}{F_n}} } = 0.\end{equation}
\end{proposition}

\begin{proof}
	We proceed similarly to the proof of Proposition \ref{prop c-tempered orth rel}. Namely, again it is enough to find a sub-sequence on which the limit holds. We define quadlinear forms $$ \Phi_1 , \Phi_2 , \ldots : V \times \overline{V} \times \overline{W} \times W \to \bbC$$ by $$ \Phi_n (v_1 , v_2 , w_1 , w_2) := \frac{\int_{F_n} \langle g v_1 , v_2 \rangle \overline{ \langle g w_1 , w_2 \rangle} \cdot dg}{\sqrt{\Mss{v_0}{v_0}{F_n}} \sqrt{\Mss{w_0}{w_0}{F_n}} }.$$ We see that these are jointly bounded, and that for all $g_1 , g_2 \in G$ and $v_1 , v_2 \in V$ and $w_1 , w_2 \in W$ we have $$ \lim_{n \to +\infty} \left( \Phi_n (g_1 v_1 , g_2 v_2 , g_1 w_1 , g_2 w_2) - \Phi_n (v_1 , v_2 , w_1 , w_2) \right) = 0.$$ We then find a bounded quadlinear form $$ \Phi : V \times \overline{V} \times \overline{W} \times W \to \bbC$$ and a sub-sequence $0 \leq m_0 < m_1 < \ldots$ such that $\lim_{n \to +\infty}^{\textnormal{weak-}*} \Phi_{m_n} = \Phi$. We get, for all $g_1 , g_2 \in G$ and $v_1 , v_2 \in V$ and $w_1 , w_2 \in W$: $$ \Phi (g_1 v_1 , g_2 v_2 , g_1 w_1 , g_2 w_2) = \Phi (v_1 , v_2 , w_1 , w_2).$$ By Schur's lemma we obtain $\Phi = 0$, giving us the desired.
\end{proof}

\subsection{}

It is easy to answer Question \ref{main question} in the case of square-integrable representations:

\begin{proposition}\label{prop sq int}
	Let $V$ be a square-integrable irreducible unitary $G$-representation. Then $V$ is c-tempered with F{\o}lner sequence any increasing sequence $F_0 , F_1 , \ldots $ of open pre-compact subsets in $G$, such that $1 \in F_0$ and $\cup_{n \ge 0} F_n = G$.
\end{proposition}

\begin{proof}
	Recall, that matrix coefficients of a square-integrable irreducible representation are square integrable. Let $v_0 \in V$ be a unit vector. Let $F_0 , F_1 , \ldots $ be any increasing sequence of open pre-compact subsets in $G$ whose union is $G$ and with $1 \in F_0$. Let $v_1,v_2 \in V$. Condition (1) of Definition \ref{def c-temp} holds because we have $$ \Mss{v_1}{v_2}{F_n} \leq \Mss{v_1}{v_2}{G} \leq \left( \frac{ \Mss{v_1}{v_2}{G} }{ \Mss{v_0}{v_0}{F_1} } \right) \cdot \Mss{v_0}{v_0}{F_n} .$$ As for condition (2) of Definition \ref{def c-temp}, let $\epsilon > 0$ and let $K \subset G$ be compact. There exists $n_0 \ge 0$ such that $$ \Mss{v_1}{v_2}{G \smallsetminus F_{n_0}} \leq \epsilon \cdot \Mss{v_0}{v_0}{F_1} .$$ There exists $n_1 \ge n_0$ such that $K F_{n_0} K^{-1} \subset F_{n_1}$. Let $n \ge n_1$ and let $g_1 , g_2 \in K$. Notice that $ ( F_n \triangle g_2^{-1} F_n g_1 ) \cap F_{n_0} = \emptyset$. Thus we have $$ \Mss{v_1}{v_2}{F_n \triangle g_2^{-1} F_n g_1} \leq \Mss{v_1}{v_2}{G \smallsetminus F_{n_0}} \leq \epsilon \cdot \Mss{v_0}{v_0}{F_1} \leq $$ $$ \leq \epsilon \cdot \Mss{v_0}{v_0}{F_n}.$$
\end{proof}

\section{c-Tempered irreps are tempered}\label{sec ctemp is temp}

In this section, let $G$ be a unimodular second countable locally compact group. We introduce some intermediate concepts, with the goal of showing that c-tempered irreducible unitary $G$-representations are tempered (Corollary \ref{cor c-temp are temp}).

\subsection{}

Let us recall some standard definitions and statements regarding weak containment.

\begin{definition}
	Let $V$ and $W$ be unitary $G$-representations.
	
	\begin{enumerate}
		\item $V$ is \textbf{weakly contained} in $W$ if for every $v \in V$, compact $K \subset G$ and $\epsilon > 0$ there exist $w_1 , \ldots , w_r \in W$ such that $$ | \langle gv , v \rangle - \sum_{1 \leq i \leq r} \langle g w_i , w_i \rangle | \leq \epsilon$$ for all $g \in K$.
		\item $V$ is \textbf{Zimmer-weakly contained}\footnote{or ``weakly contained in the sense of Zimmer", following \cite[Remark F.1.2.(ix)]{BeHaVa}.} in $W$ if for every $v_1 , \ldots , v_r \in V$, compact $K \subset G$ and $\epsilon > 0$ there exist $w_1 , \ldots , w_r \in W$ such that $$ |\langle gv_i , v_j \rangle - \langle g w_i , w_j \rangle | \leq \epsilon $$ for all $1 \leq i,j \leq r$ and $g \in K$.
	\end{enumerate}
	
\end{definition}

To facilitate the formulation of the next lemma, let us also give the following intermediate definition:

\begin{definition}
	Let $V$ and $W$ be unitary $G$-representations. Let us say that $V$ is \textbf{strongly-weakly contained} in $W$ if for every $v \in V$, compact $K \subset G$ and $\epsilon > 0$ there exists $w \in W$ such that $$ | \langle g v , v \rangle - \langle gw , w \rangle | \leq \epsilon$$ for all $g \in K$.
\end{definition}

\begin{lemma}\label{lem prop of weak cont}
	Let $V$ and $W$ be unitary $G$-representations.
	
	\begin{enumerate}
		\item If $V$ is Zimmer-weakly contained in $W$ then $V$ is strongly-weakly contained in $W$, and if $V$ is strongly-weakly contained in $W$ then $V$ is weakly contained in $W$.
		\item If $V$ is weakly contained in $W$ then $V$ is strongly-weakly contained in\footnote{Here, $W^{\oplus \infty}$ stands for the Hilbert direct sum of countably many copies of $W$.} $W^{\oplus \infty}$.
		\item If $V$ is weakly contained in $W^{\oplus \infty}$ then $V$ is weakly contained in $W$.
		\item If $V$ is irreducible and $V$ is weakly contained in $W$ then $V$ is strongly-weakly contained in $W$.
		\item If $V$ is cyclic (in particular, if $V$ is irreducible) and $V$ is strongly-weakly contained in $W$ then $V$ is Zimmer-weakly contained in $W$.
		\item If $V$ is strongly-weakly contained in $W$ then $V$ is Zimmer-weakly contained in $W^{\oplus \infty}$.
	\end{enumerate}
\end{lemma}

\begin{proof}
	Statements $(1)$, $(2)$ and $(3)$ are straight-forward. For statement $(4)$ see, for example, \cite[Proposition F.1.4]{BeHaVa}. For statement $(5)$ see \cite[proof of $(iii)\implies(iv)$ of Proposition 2.2]{Ke}. For statement $(6)$, again see \cite[proof of $(iii)\implies(iv)$ of Proposition 2.2]{Ke} (one writes $V$ as a Hilbert direct sum of countably many cyclic unitary $G$-representations, and uses item $(5)$).
\end{proof}

\begin{corollary}\label{cor temp Zimmer temp}
	Let $V$ and $W$ be unitary $G$-representations.
	
	\begin{enumerate}
		\item $V$ is weakly contained in $W$ if and only if $V$ is Zimmer-weakly contained in $W^{\oplus \infty}$.
		\item If $V$ is irreducible, $V$ is weakly contained in $W$ if and only if $V$ is Zimmer-weakly contained in $W$.
	\end{enumerate}
	
\end{corollary}

The following definition of temperedness is classical:

\begin{definition}
	A unitary $G$-representation $V$ is said to be \textbf{tempered} if $V$ is weakly contained in\footnote{Recall that $L^2 (G)$ denotes $L^2 (G , dg)$, viewed as a unitary $G$-representation via the right regular action.} $L^2 (G)$.
\end{definition}

\begin{remark}\label{rem irr zimmer}
Notice that an irreducible unitary $G$-representation is tempered if and only if it is Zimmer-weakly contained in $L^2 (G)$, by part $(2)$ of Corollary \ref{cor temp Zimmer temp}.
\end{remark}

\subsection{}

The next definitions are related to the idea that one representation is weakly contained in another if there ``almost" exists a $G$-intertwining isometric embedding from the one to the other.

\begin{definition}\label{def asymp emb}
	Let $V$ and $W$ be unitary $G$-representations. A sequence $\{ S_n \}_{n \ge 0} \subset \calB (V ; W)$ is an \textbf{asymptotic embedding} if the following conditions are satisfied:
		
		\begin{enumerate}
			\item The operators $\{ S_n \}_{n \ge 0}$ are jointly bounded, i.e. there exists $C>0$ such that $|| S_n ||^2 \leq C$ for all $n \ge 0$.
			\item Given $v_1 , v_2 \in V$ and a compact $K \subset G$ we have $$ \lim_{n \to +\infty} \ \sup_{g \in K} | \langle (S_n g - g S_n) v_1 , S_n v_2 \rangle | = 0.$$
			\item Given $v_1 , v_2 \in V$, we have $$ \lim_{n \to +\infty} \ \langle S_n v_1 , S_n v_2 \rangle = \langle v_1 , v_2 \rangle.$$
		\end{enumerate}
\end{definition}

\begin{definition} Let $V$ and $W$ be unitary $G$-representations.
	\begin{enumerate}
		\item We say that $V$ is \textbf{o-weakly contained}\footnote{``o" stands for ``operator".} in $W$ if there exists an asymptotic embedding $\{ S_n \}_{n \ge 0} \subset \calB (V ; W)$.
		\item We say that $V$ is \textbf{o-tempered} if it is o-weakly contained in $L^2 (G)$.
	\end{enumerate}
\end{definition}

\begin{lemma}\label{lem asymp emb uniform}
	In the context of Definition \ref{def asymp emb}, if conditions $(1)$ and $(2)$ of Definition \ref{def asymp emb} are satisfied then given compacts $L_1 , L_2 \subset V$ and a compact $K \subset G$ we have $$ \lim_{n \to +\infty} \sup_{v_1 \in L_1 , v_2 \in L_2 , g \in K} | \langle (S_n g - g S_n) v_1 , S_n v_2 \rangle | = 0,$$ and if conditions $(1)$ and $(3)$ of Definition \ref{def asymp emb} are satisfied then given compacts $L_1, L_2 \subset V$ we have $$ \lim_{n \to +\infty} \sup_{v_1 \in L_1 , v_2 \in L_2} |\langle S_n v_1 , S_n v_2 \rangle - \langle v_1 , v_2 \rangle| = 0.$$
\end{lemma}

\begin{proof}
	This follows from the well-known fact from functional analysis that pointwise convergence coincides with compact convergence on equi-continuous subsets, see \cite[Proposition 32.5]{Tr}.
\end{proof}

\begin{lemma}\label{lem asymp emb irrep}
	In the context of Definition \ref{def asymp emb}, assume that $V$ is irreducible. If conditions $(1)$ and $(2)$ of Definition \ref{def asymp emb} are satisfied then there exists a sub-sequence $0 \leq m_0 < m_1 < \ldots$ and $c \in \bbR_{\ge 0}$ such that for all $v_1 , v_2 \in V$ we have \begin{equation}\label{eq limit c} \lim_{n \to +\infty} \ \langle S_{m_n} v_1 , S_{m_n} v_2 \rangle = c \cdot \langle v_1 , v_2 \rangle.\end{equation} In particular, if there exists $v \in V$ such that $\liminf_{n \to +\infty} || S_n v ||^2 > 0$ then there exists $d \in \bbR_{> 0}$ (in fact, $d^{-2} = \lim_{n \to +\infty} || S_{m_n} v ||^2 / || v ||^2$) such that $\{ d S_{m_n} \}_{n \ge 0}$ satisfies condition $(3)$ of Definition \ref{def asymp emb}, i.e. is an asymptotic embedding.
\end{lemma}

\begin{proof}
	By the sequential Banach–Alaoglu theorem (applicable as $V$ is separable, and $\{ S_n^* S_n \}_{n \ge 0}$ are jointly bounded by condition $(1)$), there exists a sub-sequence $1 \leq m_0 < m_1 < \ldots$ such that $\{ S_{m_n}^* S_{m_n} \}_{n \ge 0}$ converges in the weak operator topology to some $S \in \calB (V)$.
	
	\medskip
	
	Let us first check that $S$ is $G$-invariant. For $g \in G$ and $v_1 , v_2 \in V$ we have $$ |\langle S_n^* S_n g v_1 , v_2 \rangle - \langle  S_n^* S_n v_1 , g^{-1} v_2 \rangle |  = | \langle S_n g v_1 , S_n v_2 \rangle - \langle S_n v_1 , S_n g^{-1} v_2 \rangle | \leq $$ $$ \leq |\langle (S_n g - g S_n) v_1 , S_n v _2 \rangle | + |\langle S_n v_1 , (g^{-1} S_n - S_n g^{-1}) v_2 \rangle |$$ and both summands in the last expression converge to $0$ as $n \to +\infty$ by condition $(2)$. Therefore $$ | \langle Sg v_1 , v_2 \rangle - \langle g S v_1 , v_2 \rangle | = \lim_{n \to +\infty} |\langle S_{m_n}^* S_{m_n} g v_1 , v_2 \rangle - \langle  S_{m_n}^* S_{m_n} v_1 , g^{-1} v_2 \rangle | = 0 $$ i.e. $\langle S g v_1 , v_2 \rangle = \langle g S v_1 , v_2 \rangle$. Thus, since $v_1$ and $v_2$ were arbitrary, $S g = g S$. This holds for all $g \in G$, i.e. $S$ is $G$-invariant.
	
	\medskip
	
	By Schur's lemma, we deduce $S = c \cdot \textnormal{Id}_V$ for some $c \in \bbC$. This translates precisely to (\ref{eq limit c}). The last claim is then straight-forward.
	
\end{proof}

\begin{remark}\label{rem alt cond}
	Using Lemma \ref{lem asymp emb uniform}, it is straight-forward that, assuming condition $(1)$ of Definition \ref{def asymp emb}, conditions $(2)$ and $(3)$ in Definition \ref{def asymp emb} are equivalent to the one condition that for $v_1 , v_2 \in V$ and a compact $K \subset G$ one has $$ \lim_{n \to +\infty} \sup_{g \in K} | \langle g S_n v_1 , S_n v_2 \rangle - \langle g v_1 , v_2 \rangle | = 0.$$ Indeed, let us write \begin{equation}\label{eq alt cond} \langle g S_n v_1 , S_n v_2 \rangle - \langle gv_1 , v_2 \rangle = \langle (g S_n - S_n g) v_1 , S_n v_2 \rangle + ( \langle S_n g v_1 , S_n v_2 \rangle - \langle gv_1 , v_2 \rangle).\end{equation} The current condition gives condition $(3)$ by plugging in $g = 1$, and then (\ref{eq alt cond}) gives condition $(2)$, using the uniformity provided by Lemma \ref{lem asymp emb uniform}. Conversely, (\ref{eq alt cond}) shows  immediately (again taking into consideration Lemma \ref{lem asymp emb uniform}) that conditions $(2)$ and $(3)$ imply the current condition.
\end{remark}

\subsection{}

The concept of o-weak containment in fact coincides with that of Zimmer-weak containment:

\begin{proposition}\label{prop o cont is weakly cont}
	Let $V$ and $W$ be unitary $G$-representations. Then $V$ is o-weakly contained in $W$ if and only if $V$ is Zimmer-weakly contained in $W$.
\end{proposition}

\begin{proof}
	Let $\{ S_n \}_{n \ge 0} \subset \calB (V ; W)$ be an asymptotic embedding. Given $v_1 , \ldots , v_r \in V$, by Remark \ref{rem alt cond}, given any compact $K \subset G$ we have $$ \lim_{n \to +\infty} \sup_{g \in K} |\langle g S_n v_i , S_n v_j \rangle - \langle g v_i , v_j \rangle | = 0$$ for all $1 \leq i,j \leq r$, and thus $$ \lim_{n \to +\infty} \sup_{g \in K} \sup_{1 \leq i,j \leq r} |\langle g S_n v_i , S_n v_j \rangle - \langle g v_i , v_j \rangle | = 0.$$ Thus by definition $V$ is Zimmer-weakly contained in $W$.
	
	\medskip
	
	Conversely, suppose that $V$ is Zimmer-weakly contained in $W$. Let $\{ e_n \}_{n \ge 0}$ be an orthonormal basis for $V$. Let $\{ K_n \}_{n \ge 0}$ be an increasing sequence of compact subsets in $G$, with $1 \in K_0$ and with the property that for any compact subset $K \subset G$ there exists $n \ge 0$ such that $K \subset K_n$. As $V$ is Zimmer-weakly contained in $W$, given $n \ge 0$, let us find $w^n_0 , \ldots , w^n_n \in W$ such that $$ \sup_{g \in K_n} |\langle g e_i , e_j \rangle - \langle g w^n_i , w^n_j \rangle | \leq \frac{1}{n+1}$$ for all $0 \leq i,j \leq n$. Define $S_n : V \to W$ by $$ S_n \left( \sum_{i \ge 0} c_i \cdot e_i \right) := \sum_{0 \leq i \leq n} c_i \cdot w^n_i.$$ We want to check that $\{ S_n \}_{n \ge 0}$ is an asymptotic embedding. As for condition $(1)$, notice that $$ \left\Vert S_n \left( \sum_{i \ge 0} c_i e_i \right) \right\Vert^2 = \left\Vert \sum_{0 \leq i \leq n} c_i w^n_i \right\Vert^2 = \left| \sum_{0 \leq i,j \leq n} c_i \overline{c_j} \cdot \langle w^n_i , w^n_j \rangle \right| \leq $$ $$ \leq  \left| \sum_{0 \leq i,j \leq n} c_i \overline{c_j} \cdot \langle e_i , e_j \rangle \right| + \left| \sum_{0 \leq i,j \leq n } c_i \overline{c_j} \cdot \left( \langle w_i^n , w_j^n \rangle - \langle e_i , e_j \rangle \right)\right|  \leq $$ $$ \leq \sum_{0 \leq i \leq n} |c_i|^2 + \frac{1}{n+1} \left( \sum_{0 \leq i \leq n} |c_i| \right)^2 \leq2  \sum_{0 \leq i \leq n} |c_i|^2 \leq 2 \cdot \left\Vert \sum_{i \ge 0} c_i e_i \right\Vert^2,$$ showing that $|| S_n ||^2 \leq 2$ for all $n \ge 0$. It is left to show the condition as in Remark \ref{rem alt cond}. Let us thus fix a compact $K \subset G$. Notice that it is straight-forward to see that it is enough to check the condition for vectors in a subset of $V$, the closure of whose linear span is equal to $V$. So it is enough to check that $$ \lim_{n \to +\infty} \sup_{g \in K} |\langle g S_n e_i , S_n e_j \rangle - \langle g e_i , e_j \rangle | = 0$$ for any given $i,j \ge 0$. Taking $n$ big enough so that $K \subset K_n$ and $n \ge \max \{ i,j \}$, we have $$ \sup_{g \in K} |\langle g S_n e_i , S_n e_j \rangle - \langle g e_i , e_j \rangle | = \sup_{g \in K} |\langle g w^n_i , w^n_j \rangle - \langle g e_i , e_j \rangle | \leq \frac{1}{n+1},$$ giving the desired.
\end{proof}

\begin{corollary}\label{cor o temp is temp}
	An irreducible unitary $G$-representation is o-tempered if and only if it is tempered.
\end{corollary}

\begin{proof}
	This is a special case of Proposition \ref{prop o cont is weakly cont}, taking into account Remark \ref{rem irr zimmer}.
\end{proof}

\subsection{}

Here we give a weaker version of c-temperedness, which is technically convenient to relate to other concepts of this section.

\begin{definition}\label{def right-c-temp}
	Let $V$ be an irreducible unitary $G$-representation. Let $F_0 , F_1 , \ldots \subset G$ be a sequence of measurable pre-compact subsets all containing a neighbourhood of $1$. We say that $V$ is \textbf{right-c-tempered with F{\o}lner sequence $F_0 , F_1 , \ldots$} if there exists a unit vector $v_0 \in V$ such that the following two conditions are satisfied:
	\begin{enumerate}
		\item For all $v \in V$ we have $$ \limsup_{n \to +\infty} \frac{ \Mss{v}{v_0}{F_n} }{ \Mss{v_0}{v_0}{F_n} } < +\infty.$$
		\item For all $v \in V$ and all compact subsets $K \subset G$ we have $$ \lim_{n \to +\infty} \frac{\sup_{g \in K} \Mss{v}{v_0}{F_n \triangle F_n g} }{ \Mss{v_0}{v_0}{F_n} } = 0.$$
	\end{enumerate}
\end{definition}

\subsection{}

Finally, we can show that c-tempered irreducible unitary $G$-representations are tempered.

\begin{proposition}\label{prop right-c is op}
	Let $V$ be an irreducible unitary $G$-representation. Assume that $V$ is right-c-tempered (with some F{\o}lner sequence). Then $V$ is o-tempered. More precisely, suppose that $V$ is right-c-tempered with F{\o}lner sequence $F_0 , F_1 , \ldots$ and let $v_0 \in V$ be a unit vector for which the conditions (1) and (2) of Definition \ref{def right-c-temp} are satisfied. Then the sequence of operators $$ S_0 , S_1 , \ldots : V \to L^2 (G)$$ given by $$ S_n (v) (x) := \begin{cases} \frac{1}{\sqrt{ \Mss{v_0}{v_0}{F_n} }} \cdot \langle x v , v_0 \rangle , & x \in F_n \\ 0, & x \notin F_n \end{cases}$$ admits a sub-sequence which is an asymptotic embedding.
\end{proposition}

\begin{corollary}\label{cor c-temp are temp}
	Every c-tempered irreducible unitary $G$-representation (with some F{\o}lner sequence) is tempered.
\end{corollary}

\begin{proof}
	It is clear that c-temperedness implies right-c-temperedness, Proposition \ref{prop right-c is op} says that right-c-temperedness implies o-temperedness, and Corollary \ref{cor o temp is temp} says that o-temperedness is equivalent to temperedness.
\end{proof}

\begin{proof}[Proof (of Proposition \ref{prop right-c is op}).]

	Clearly each $S_n$ is bounded. By condition (2) of Definition \ref{def right-c-temp}, for any fixed $v \in V$ there exists $C>0$ such that $||S_n (v)||^2 \leq C$ for all $n$. By the Banach-Steinhaus theorem, this implies that the operators $S_0 , S_1 , \ldots$ are jointly bounded, thus condition $(1)$ of Definition \ref{def asymp emb} is verified.
	
	\medskip
	
	To verify condition $(2)$ of Definition \ref{def asymp emb}, fix $v \in V$ and a compact $K \subset G$. Given $g \in K$ and a function $f \in L^2 (G)$ of $L^2$-norm one, we have $$ | \langle S_n (g v) - g S_n (v) , f \rangle | = \frac{ \left| \int_{F_n} \langle x g v , v_0 \rangle \overline{f(x)} \cdot dx - \int_{F_n g^{-1}} \langle xg v , v_0 \rangle \overline{f(x)} \cdot dx \right| }{\sqrt{ \Mss{v_0}{v_0}{F_n} }} \leq $$ $$ \leq \frac{ \int_{F_n \triangle F_n g^{-1}} \left| \langle xg v , v_0 \rangle  \overline{f(x)} \right| \cdot dx }{\sqrt{ \Mss{v_0}{v_0}{F_n} }} \leq \sqrt{\frac{\int_{F_n \triangle F_n g^{-1}} |\langle xg v , v_0 \rangle |^2 \cdot dx}{ \Mss{v_0}{v_0}{F_n} }} \cdot \sqrt{\int_{G} |f(x)|^2 \cdot dx} = $$ $$ = \sqrt{\frac{ \Mss{v}{v_0}{F_n g \triangle F_n} }{ \Mss{v_0}{v_0}{F_n} }}.$$ Since $f$ was arbitrary, this implies $$ || S_n (gv) - g S_n (v) || \leq \sqrt{\frac{ \Mss{v}{v_0}{F_n g \triangle F_n} }{ \Mss{v_0}{v_0}{F_n} }}$$ for $g \in K$. By condition (2) of Definition \ref{def right-c-temp}, this tends to $0$ as $n \to +\infty$, uniformly in $g \in K$, and hence the desired.
	
	\medskip
	
	Now, using Lemma \ref{lem asymp emb irrep} we see that some sub-sequence will satisfy condition $(3)$ of Definition \ref{def asymp emb}, once we notice that $|| S_n v_0 ||^2 = 1$ for all $n$ by construction.
	
\end{proof}

\section{The case of $K$-finite vectors}\label{sec Kfin}

In this section $G$ is a semisimple group over a local field. We continue with notations from \S\ref{sec intro}. The purpose of this section is to prove Theorem \ref{thm main Kfin}.

\subsection{}

Let us first show that, when $G$ is non-Archimedean, it is enough to establish condition (2) of Theorem \ref{thm main Kfin}, and condition (1) will then follow. So we assume condition (2) and use the notation $C(v_1 , v_2)$ therein.

\medskip

Let us denote by $\underline{V} \subset V$ the subspace of $K$-finite (i.e. smooth) vectors. By the polarization identity, it is clear that for all $v_1 , v_2 , v_3 , v_4 \in \underline{V}$ the limit $$ \lim_{r \to +\infty} \frac{\int_{G_{<r}} \langle gv_1 , v_2 \rangle \overline{\langle gv_3 , v_4 \rangle} \cdot dg}{r^{\mathbf{d} (V)}}$$ exists, let us denote it by $D(v_1 , v_2 , v_3 , v_4)$, and $D$ is a quadlinear form $$D : \underline{V} \times \overline{\underline{V}}  \times  \overline{\underline{V}}  \times \underline{V} \to \bbC.$$

\medskip

Next, we claim that for all $v_1 , v_2 , v_3 , v_4 \in \underline{V}$ and all $g_1 , g_2 \in G$ we have $$ D(g_1 v_1, g_2 v_2 , g_1 v_3 , g_2 v_4) = D(v_1 , v_2 , v_3 , v_4).$$ Indeed, again by the polarization identity, it is enough to show that for all $v_1 , v_2 \in \underline{V}$ and all $g_1 , g_2 \in G$ we have \begin{equation}\label{eq Kfin nonArch C is G inv} C(g_1 v_1 , g_2 v_2) = C(v_1 , v_2).\end{equation} There exists $r_0 \ge 0$ such that $$ G_{<r-r_0} \subset g_2^{-1} G_{<r} g_1 \subset G_{<r+r_0}.$$ We have: $$ \int_{G_{<r}} |\langle g g_1 v_1 , g_2 v_2 \rangle |^2 \cdot dg = \int_{g_2^{-1} G_{<r} g_1} |\langle gv_1 , v_2 \rangle |^2 \cdot dg$$ and therefore $$ \int_{G_{<r-r_0}} |\langle g v_1 , v_2 \rangle |^2 \cdot dg \leq \int_{G_{<r}} |\langle g g_1 v_1 , g_2 v_2 \rangle |^2 \cdot dg \leq \int_{G_{<r+r_0}} |\langle gv_1 , v_2 \rangle |^2 \cdot dg.$$ Dividing by $r^{\mathbf{d} (V)}$ and taking the limit $r \to +\infty$ we obtain (\ref{eq Kfin nonArch C is G inv}).

\medskip

Now, by Schur's lemma (completely analogously to the reasoning with $\Phi$ in the proof of Proposition \ref{prop c-tempered orth rel}), we obtain that for some $C >0$ we have $$ D(v_1 , v_2 , v_3 , v_4) = C \cdot \langle v_1 , v_ 3 \rangle \overline{\langle v_2 , v_4 \rangle}$$ for all $v_1 , v_2 , v_3 , v_4 \in \underline{V}$.

\subsection{}

Thus, we aim at establishing condition (2) of Theorem \ref{thm main Kfin} in either the non-Archimedean or the Archimedean cases. Since a complex group can be considered as a real group and the formulation of the desired theorem will not change, we assume that we are either in the real case or in the non-Archimedean case.

\medskip

Also, notice that to show Theorem \ref{thm main Kfin} for all maximal compact subgroups it is enough to show it for one maximal compact subgroup (in the non-Archimedean case because the resulting notion of $K$-finite vectors does not depend on the choice of $K$ and in the real case since all maximal compact subgroups are conjugate).

\subsection{}

Let us fix some notation. We choose a maximal split torus $A \subset G$ and a minimal parabolic $P \subset G$ containing $A$. We denote $$ \fraka := \Hom_{\bbZ} (X^* (A) , \bbR).$$ We let $L \subset \fraka$ to be $\fraka$ itself in the real case and the lattice in $\fraka$ corresponding to $X_* (A)$ in the non-Archimedean case. We let $\exp : L \to A$ be the exponential map constructed in the usual way:
\begin{itemize}
	\item If $G$ is real, we let $\exp$ to be the composition $L = \fraka \cong {\rm Lie} (A) \to A$ where the last map is the exponential map from the Lie algebra to the Lie group, while the isomorphism is the identification resulting from the map $X^* (A) \to {\rm Lie} (A)^*$ given by taking the differential at $1 \in A$.
	\item If $G$ is non-Archimedean, we let $\exp$ be the composition $L \cong X_* (A) \to A$ where the last map is given by sending $\chi$ to $\chi (\varpi^{-1})$, where $\varpi$ is a uniformizer.
\end{itemize}
We denote by $$ \Delta \subset \widetilde{\Delta} \subset X^* (A) \subset \fraka^*$$ the set of simple roots $\Delta$ and the set of positive roots $\widetilde{\Delta}$ (resulting from the choice of $P$). We identify $\fraka$ with $\bbR^{\Delta}$ in the clear way. We set $$ \fraka^+ := \{ x \in \fraka \ | \ \alpha (x) \ge 0 \ \forall \alpha \in \Delta \}$$ and $L^+ := L \cap \fraka^+$.

\medskip

Let us in the standard way choose a maximal comapct subgroup $K \subset G$ ``in good relative position" with $A$. In the real case this means ${\rm Lie} (A)$ sitting in the $(-1)$-eigenspace of a Cartan involution whose $1$-eigenspace is ${\rm Lie} (K)$ and in the non-Archimedean case it is as in \cite[V.5.1., Th{\' e}or{\` e}me]{Re}. In the non-Archimedean case let us also, to simplify notation, assume that $G = K \exp (L^+) K$ (in general there is a finite subset $S \subset Z_G (A)$ such that $G = \coprod_{s \in S} K \exp (L^+) s K$ and one proceeds with the obvious modifications).

\medskip

Let us denote $\rho := \frac{1}{2}\sum_{\alpha \in \widetilde{\Delta}} \mu_{\alpha} \cdot \alpha \in \fraka^*$ where $\mu_{\alpha} \in \bbZ_{\ge 1}$ is the multiplicity of the root $\alpha$.

\medskip

Fixing Haar measures, especially denoting by $dx$ a Haar measure on $L$, we have a uniquely defined continuous $\omega : L^+ \to \bbR_{\ge 0}$ such that the following integration formula holds: $$ \int_{G} f(g) \cdot dg = \int_{K \times K} \left( \int_{L^+} \omega (x) f(k_1 \exp (x) k_2) \cdot dx \right) \cdot  dk_1 dk_2.$$ Regarding the behaviour of $\omega (x)$, we can use \cite[around Lemma 1.1]{Ar} as a reference. In the real case there exists $C>0$ such that
\begin{equation}\label{eq omega is like rho Arch}
\frac{\omega (x)}{e^{2\rho (x)}} = C \cdot \prod_{\alpha} \left( 1 - e^{-2\alpha (x)} \right)
\end{equation} where $\alpha$ runs over $\widetilde{\Delta}$ according to multiplicities $\mu_{\alpha}$. In the non-Archimedean case, for every $\Theta \subset \Delta$ there exists $C_{\Theta} > 0$ such that
\begin{equation}\label{eq omega is like rho nonArch}
\frac{\omega (x)}{e^{2\rho (x)}} = C_{\Theta}
\end{equation} for all $x \in L^+$ satisfying $\alpha (x) = 0$ for all $\alpha \in \Theta$ and $\alpha (x) \neq 0$ for all $\alpha \in \Delta \smallsetminus \Theta$.

\medskip

Since, by Claim \ref{clm doesnt depend on norm}, we are free in our choice of the norm $||-||$ on $\frakg$, let us choose $||-||$ to be a supremum norm in coordinates gotten from an $A$-eigenbasis. Then\footnote{Recall the notation $\mathbf{r}$ from \S\ref{sec intro}.} $$ \mathbf{r} (\exp (x)) = \log q \cdot \max_{\alpha \in \widetilde{\Delta} } |\alpha (x)|$$ where $q$ is the residual cardinality in the non-Archimedean case and $q := e$ in the real case. Let us denote $$ \fraka_{<r} := \{ x \in \fraka \ | \ | \alpha (x)| < \tfrac{r}{\log q} \ \forall \alpha \in \widetilde{\Delta} \}$$ and $\fraka^+_{<r} := \fraka_{<r} \cap \fraka^+$ and similarly $L_{<r} := L \cap \fraka_{<r}$, $L^+_{<r} := L^+ \cap L_{<r}$. Then $ L_{<r} = \exp^{-1} (G_{<r})$. Hence there exists $r_0 \ge 0$ such that
\begin{equation}\label{eq Gr is like Lr}
K \exp (L^+_{<r-r_0}) K \subset G_{<r} \subset K \exp(L^+_{<r+r_0}) K.
\end{equation}

\subsection{}

Let now $V$ be a tempered irreducible unitary $G$-representation. Let us denote by $\underline{V} \subset V$ the subspace of $K$-finite vectors. Given $v_1 , v_2 \in V$, we will denote by $f_{v_1 , v_2}$ the continuous function on $L^+$ given by $$f_{v_1 , v_2} (x) := e^{\rho (x)} \langle \exp (x) v_1 , v_2 \rangle.$$ We have $$ \Ms{v_1}{v_2}{r} = \int_{G_{<r}} |\langle gv_1 , v_2 \rangle |^2 \cdot dg = $$ $$= \int_{K \times K} \left( \int_{L^+ \cap \exp^{-1} (k_2 G_{<r} k_1^{-1})} \frac{\omega (x)}{e^{2 \rho (x)}} |f_{k_1 v_1 , k_2 v_2} (x)|^2 \cdot dx \right) \cdot dk_1 dk_2.$$ In view of (\ref{eq Gr is like Lr}), in order to prove Theorem \ref{thm main Kfin} it is enough to show:

\begin{claim}\label{clm Kfin ess}
	There exists $\mathbf{d} (V) \in \bbZ_{\ge 0}$ such that for every non-zero $v_1 , v_2 \in \underline{V}$ there exists $C(v_1 , v_2)>0$ such that $$ \lim_{r \to +\infty} \frac{\int_{K \times K} \left( \int_{L^+_{<r}} \frac{\omega (x)}{e^{2 \rho (x)}}| f_{k_1 v_1 , k_2 v_2} (x) |^2 \cdot dx \right) \cdot dk_1 dk_2}{r^{\mathbf{d} (V)}} = C(v_1 , v_2).$$
\end{claim}

\subsection{} We have the following:

\begin{claim}\label{clm growth exists}\
	\begin{enumerate}
		\item Given $v_1 , v_2 \in \underline{V}$, either $f_{v_1 , v_2} = 0$ in which case we set $\mathbf{d} (v_1 , v_2) := -\infty$, or there exist $\mathbf{d} (v_1 , v_2) \in \bbZ_{\ge 0}$ and $C(v_1 , v_2)>0$ such that $$ \lim_{r \to +\infty} \frac{1}{r^{\mathbf{d} (v_1 , v_2)}} \int_{L^+_{<r}} \frac{\omega (x)}{e^{2 \rho (x)}} |f_{v_1 , v_2} (x)|^2 \cdot dx = C(v_1 , v_2).$$
		\item In the real case, we have $\mathbf{d} (v_1 , X v_2) \leq \mathbf{d} (v_1 , v_2)$ for all $v_1 , v_2 \in \underline{V}$ and $X \in \frakg$.
		\item Denoting $\mathbf{d}(V) := \sup_{v_1 , v_2 \in \underline{V}} \mathbf{d} (v_1 , v_2)$, we have neither $\mathbf{d} (V) = -\infty$ nor $\mathbf{d} (V) =+\infty$ (i.e. $\mathbf{d} (V) \in \bbZ_{\ge 0}$).
	\end{enumerate}
\end{claim}

Let us establish Claim \ref{clm Kfin ess} given Claim \ref{clm growth exists}:

\begin{proof}[Proof (of Claim \ref{clm Kfin ess} given Claim \ref{clm growth exists}).]
	Let us first handle the non-Archimedean case. Let us notice that we can replace $v_1$ and $v_2$ by $g_1 v_1$ and $g_2 v_2$ for any $g_1,g_2 \in G$. Indeed, for some $r_0 \ge 0$ we have $$g_2^{-1} G_{<r-r_0} g_1 \subset G_{<r} \subset g_2^{-1} G_{<r+r_0} g_1$$ and thus $$ \int_{G_{<r-r_0}} | \langle g g_1 v_1 , g_2 v_2 \rangle |^2 \cdot dg \leq \int_{G_{<r}} |\langle g v_1 , v_2 \rangle |^2 \cdot dg \leq \int_{G_{<r+r_0}} | \langle g g_1 v_1 , g_2 v_2 \rangle |^2 \cdot dg,$$ from which the claim clearly follows. Since $G \cdot v_1$ spans $\underline{V}$ and $G \cdot v_2$ spans $\underline{V}$, we deduce that by replacing $v_1$ and $v_2$ we can assume that $\mathbf{d} (v_1 , v_2) = \mathbf{d} (V)$. Now, since the integral $$ \int_{K \times K} \left( \int_{L^+_{<r}} \frac{\omega (x)}{e^{2 \rho (x)}} |f_{k_1 v_1, k_2 v_2} (x)|^2 \cdot dx \right) \cdot dk_1 dk_2$$ over $K \times K$ is simply a finite linear combination the claim is clear.
	
	\medskip
	
	Let us now handle the real case. First, we would like to see that for some $k_1 , k_2 \in K$ we have $\mathbf{d} (k_1 v_1 , k_2 v_2) = \mathbf{d} (V)$. To that end, let us denote by $\frakn$ and $\frakn^-$ the Lie algebras of $N$ and $N^-$ (the unipotent radicals of $P$ and of $P^-$, the opposite to $P$ with respect to $A$) and identify $\fraka$ with the Lie algebra of $A$ as before. Since $\calU (\frakn^-) \calU (\fraka) K v_1$ spans $\underline{V}$ and $\calU (\frakn) \calU (\fraka) K v_2$ spans $\underline{V}$, we can find $k_1 , k_2 \in K$ and some elements $v'_1 \in \calU (\frakn^-) \calU (\fraka) k_1 v_1$ and $v'_2\in \calU (\frakn) \calU (\fraka) k_2 v_2$ such that $\mathbf{d} (v'_1, v'_2) = \mathbf{d} (V)$. By Claim \ref{clm growth exists}(2) this forces $\mathbf{d} (k_1 v_1 , k_2 v_2) = \mathbf{d} (V)$. 
	
	\medskip
	
	Next, given two continuous functions $f_1 , f_2$ on $\fraka^+$ and $d \in \bbZ_{\ge 0}$ let us denote $$ \langle f_1 , f_2 \rangle_d := \lim_{r \to +\infty} \frac{ \int_{\fraka^+_{<r}} \frac{\omega (x)}{e^{2 \rho (x)}} f_1(x) \overline{f_2 (x)} \cdot dx}{r^d}$$ if the limit exists, and $|| f ||^2_d := \langle f , f \rangle_d$.
	
	\medskip
	
	We claim that the function $(k_1 , k_2) \mapsto || f_{k_1 v_1 , k_2 v_2} ||^2_{\mathbf{d} (V)}$ on $K \times K$ is continuous and that $$ \lim_{r \to +\infty} \frac{\int_{K \times K} \left( \int_{\fraka^+_{<r}} \frac{\omega (x)}{e^{2 \rho (x)}} |f_{k_1 v_1 , k_2 v_2} (x)|^2 \cdot dx \right) \cdot dk_1 dk_2}{r^{\mathbf{d} (V)}} = \int_{K \times K} || f_{k_1 v_1 , k_2 v_2 } ||^2_{\mathbf{d} (V)} \cdot dk_1 dk_2.$$ Then the right hand side is non-zero since we have seen that $\mathbf{d} (k_1 v_1 , k_2 v_2) = \mathbf{d} (V)$ for some $k_1 , k_2 \in V$, and we are done.
	
	\medskip

	Let $(v_1^i)$ be a basis for the $\bbC$-span of $\{ k v_1 \}_{k \in K}$ and let $(v_2^j)$ be a basis for the $\bbC$-span of $\{ k v_2 \}_{k \in K}$. Let us write $k v_1 = \sum_i c_i(k) v_1^i$ and $k v_2 = \sum_j d_j (k) v_2^j$, so that $c_i$ and $d_j$ are continuous $\bbC$-valued functions of $K$. Then $$ \int_{\fraka^+_{<r}} \frac{\omega (x)}{e^{2 \rho (x)}} |f_{k_1 v_1 , k_2 v_2} (x)|^2 \cdot dx = $$ $$ = \sum_{i_1 , i_2 , j_1 , j_2} c_{i_1} (k_1) \overline{c_{i_2} (k_1)} \overline{d_{j_1} (k_2)} d_{j_2} (k_2) \int_{\fraka^+_{<r} } \frac{\omega (x)}{e^{2 \rho (x)}} \cdot f_{v_1^{i_1} , v_2^{j_1}} (x) \cdot \overline{f_{v_1^{i_2} , v_2^{j_2}} (x)} \cdot dx.$$ Therefore $$ || f_{k_1 v_1 , k_2 v_2}||^2_{\mathbf{d} (V)} = \sum_{i_1 , i_2 , j_1 , j_2} c_{i_1} (k_1) \overline{c_{i_2} (k_1)} \overline{d_{j_1} (k_2)} d_{j_2} (k_2) \langle f_{v_1^{i_1} , v_2^{j_1}} , f_{v_1^{i_2} , v_2^{j_2}} \rangle_{\mathbf{d} (V)} $$ so $(k_1 , k_2) \mapsto || f_{k_1 v_1 , k_2 v_2} ||^2_{\mathbf{d} (V)}$ is indeed continuous. Also, it is now clear that we have $$  \lim_{r \to +\infty} \frac{\int_{K \times K} \left( \int_{\fraka^+_{<r}} \frac{\omega (x)}{e^{2 \rho (x)}} |f_{k_1 v_1 , k_2 v_2} (x)|^2 \cdot dx \right) \cdot dk_1 dk_2}{r^{\mathbf{d} (V)}} = $$ $$ = \sum_{i_1 , i_2 , j_1 , j_2} \langle f_{v_1^{i_1} , v_2^{j_1}} , f_{v_1^{i_2} , v_2^{j_2}} \rangle_{\mathbf{d} (V)} \int_{K \times K} c_{i_1} (k_1) \overline{c_{i_2} (k_1)} \overline{d_{j_1} (k_2)} d_{j_2} (k_2) \cdot dk_1 dk_2 = $$ $$ = \int_{K \times K} || f_{k_1 v_1 , k_2 v_2} ||^2_{\mathbf{d} (V) } \cdot  dk_1 dk_2$$
\end{proof}

\subsection{}

Let us now explain Claim \ref{clm growth exists} in the case when $G$ is non-Archimedean. Let $v_1 , v_2 \in \underline{V}$. Let us choose a positive integer $k$ large enough so that $k \cdot \bbZ_{\ge 0}^{\Delta} \subset L$. By enlarging $k$ even more if necessary, by \cite[Theorem 4.3.3.]{Ca} for every  $\Theta \subset \Delta$ and every $y \in ( \bbR_{\ge 0}^{\Theta} \times \bbR_{>k}^{\Delta \smallsetminus \Theta}) \cap  L^+ $ the function $$ f_{v_1 , v_2 , \Theta , y} : k \cdot \bbZ_{\ge 0}^{\Delta \smallsetminus \Theta} \to \bbC$$ given (identifying $\bbR^{\Delta \smallsetminus \Theta}$ with a subspace of $\bbR^{\Delta}$ in the clear way) by $x \mapsto f_{v_1 , v_2}(y+x)$, can be written as $$ \sum_{1 \leq i \leq p} c_i \cdot e^{\lambda_i (x_{\Delta \smallsetminus \Theta})} q_i (x_{\Delta \smallsetminus \Theta})$$ where $c_i \in \bbC \smallsetminus \{ 0\}$, $\lambda_i$ is a complex-valued functional on $\bbR^{\Delta \smallsetminus \Theta}$, and $q_i$ is a monomial on $\bbR^{\Delta \smallsetminus \Theta}$. Here $x_{\Delta \smallsetminus \Theta}$ is the image of $x$ under the natural projection $\bbR^{\Delta} \to \bbR^{\Delta \smallsetminus \Theta}$. We can assume that the couples in the collection $\{ (\lambda_i , q_i) \}_{1 \leq i \leq p}$ are pairwise different. Since $V$ is tempered, by ``Casselman's criterion" we in addition have that for every $1 \leq i \leq p$, ${\rm Re} (\lambda_i)$ is non-negative on $\bbR_{\ge 0}^{\Delta \smallsetminus \Theta}$.

\medskip

By Claim \ref{clm appb lat}, either $p = 0$, equivalently $f_{v_1 , v_2 , \Theta , y} = 0$ (in which case we set $d_{v_1 , v_2 , \Theta , y} := -\infty$), or there exists $d_{v_1 , v_2 , \Theta , y} \in \bbZ_{\ge 0}$ such that the limit $$ \lim_{r \to +\infty} \frac{1}{r^{d_{v_1 , v_2 , \Theta , y}}} \sum_{x \in (k \cdot \bbZ_{\ge 0}^{\Delta \smallsetminus \Theta}) \cap L^+_{<r}} |f_{v_1 , v_2}(y+x)|^2$$ exists and is strictly positive.

\medskip

Now, given $y \in L^+$ let us denote $\Theta_y := \{ \alpha \in \Delta \ | \ y_{\alpha} \leq k \}$ where by $y_{\alpha}$ we denote the coordinate of $y \in \bbR^{\Delta}$ at the $\alpha$-place. Let $Y \subset L^+$ be the subset of $y \in L^+$ for which $y_{\alpha} \leq 2k$ for all $\alpha \in \Delta$. Then $Y$ is a finite set, and we have \begin{equation}\label{eq break L into subsets} L^+ = \coprod_{y \in Y} \left( y + k \cdot \bbZ_{\ge 0}^{\Delta \smallsetminus \Theta_y} \right).\end{equation} Notice also that $\omega (x) / e^{2 \rho (x)}$ is a positive constant on each one of the subset of which we take union in (\ref{eq break L into subsets}). We set $\mathbf{d} (v_1 , v_2) := \max_{y \in Y} d_{v_1 , v_2 , \Theta_y , y}$. We see that either $f_{v_1 , v_2} = 0$ (then $\mathbf{d} (v_1 , v_2) = -\infty$) or the limit $$ \lim_{r \to +\infty} \frac{1}{r^{\mathbf{d} (v_1 , v_2)}} \sum_{x \in L^+_{<r} } \frac{\omega (x)}{e^{2 \rho (x)}} |f_{v_1 , v_2} (x)|^2$$ exists and is strictly positive.

\medskip

That $\mathbf{d} (V)$ is finite follows from $\mathbf{d} (v_1 , v_2)$ being controlled by finitely many Jacquet modules, with the finite central actions on them.

\subsection{}

Let us now explain Claim \ref{clm growth exists} in the case when $G$ is real. Using \cite{CaMi} we know that, fixing $k >0$, given $\Theta \subset \Delta$ the restriction of $\frac{\omega (x)^{1/2}}{e^{\rho (x)}} f_{v_1 , v_2} (x)$ to $\bbR_{\ge k}^{\Delta \smallsetminus \Theta} \times [0,k]^{\Theta}$ can be written as $$ \sum_{1 \leq i \leq p} e^{\lambda_i (x_{\Delta \smallsetminus \Theta})} q_i (x_{\Delta \smallsetminus \Theta}) \phi_i (x) $$ where the notation is as follows. First, $\lambda_i$ is a complex-valued functional on $\bbR^{\Delta \smallsetminus \Theta}$. Next, $q_i$ is a monomial on $\bbR^{\Delta \smallsetminus \Theta}$. The couples $(\lambda_i , q_i)$, for $1 \leq i \leq p$, are pairwise distinct. The function $\phi_i$ is expressible as a composition $$ [0,k]^{\Theta} \times \bbR_{\ge k}^{\Delta \smallsetminus \Theta} \xrightarrow{{\rm id} \times {\rm ei}} [0,k]^{\Theta} \times (\bbC_{|-|<1})^{\Delta \smallsetminus \Theta} \xrightarrow{\phi_i^{\circ}} \bbC$$ where ${\rm ei}$ is the coordinate-wise application of $x \mapsto e^{-x}$ and $\phi_{i}^{\circ}$ is a continuous function such that, for every $b \in [0,k]^{\Theta}$, the restriction of $\phi_i^{\circ}$ via $(\bbC_{|-|<1})^{\Delta \smallsetminus \Theta} \xrightarrow{z \mapsto (b,z)} [0,k]^{\Theta} \times (\bbC_{|-|<1})^{\Delta \smallsetminus \Theta}$ is holomorphic. Lastly, the function $b \mapsto \phi_{i}^{\circ} (b , \{0\}^{\Delta \smallsetminus \Theta} ) $ on $[0,k]^{\Theta}$ is not identically zero. Since $V$ is tempered, by ``Casselman's criterion" we in addition have that for every $1 \leq i \leq p$, ${\rm Re} (\lambda_i)$ is non-negative on $\bbR_{\ge 0}^{\Delta \smallsetminus \Theta}$.

\medskip

If $p = 0$, we set $d_{v_1 , v_2 , \Theta} := -\infty$. Otherwise, Claim \ref{clm appb} provides a number $d_{v_1 , v_2 , \Theta} \in \bbZ_{\ge 0}$, described concretely in terms of $\{ (\lambda_i , q_{i} , \phi_{i} ) \}_{1 \leq i \leq p}$, such that the limit $$ \lim_{r \to +\infty} \frac{1}{r^{d_{v_1 , v_2 , \Theta}}} \int_{ ( [0,k]^{\Theta} \times \bbR_{\ge k}^{\Delta \smallsetminus \Theta} ) \cap \fraka^+_{<r}} \frac{\omega (x)}{e^{2 \rho (x)}} |f_{v_1 , v_2} (x)|^2 \cdot dx $$ exists and is strictly positive. We set $\mathbf{d} (v_1 , v_2) := \max_{\Theta \subset \Delta} d_{v_1 , v_2 , \Theta}$. Then either $f_{v_1 , v_2} = 0$ or $\mathbf{d} (v_1 , v_2) > 0$ and the limit $$ \lim_{r \to +\infty} \frac{1}{r^{\mathbf{d} (v_1 , v_2)}} \int_{\fraka^+_{<r}} \frac{\omega (x)}{e^{2 \rho (x)}} |f_{v_1 , v_2} (x)|^2 \cdot dx$$ exists and is strictly positive. That $\mathbf{d} (V)$ is finite follows from $\mathbf{d} (v_1 , v_2)$ being controlled by finitely many data, as in \cite{CaMi}. Part (2) of Claim \ref{clm growth exists} follows easily from the concrete description of $d_{v_1 , v_2 , \Theta}$ in Claim \ref{clm appb}.

\section{Proofs for Remark \ref{rem counterexample}, Remark \ref{rem doesnt depend on norm} , Proposition \ref{prop formula ch}, Proposition \ref{prop from folner to exact} and Remark \ref{rem ctemp red is temp}.}\label{sec clms red}

In this section, $G$ is a semisimple group over a local field. We continue with notations from \S\ref{sec intro}. We explain Remark \ref{rem counterexample} (in Claim \ref{clm counterexample}), explain Remark \ref{rem doesnt depend on norm} (in Claim \ref{clm doesnt depend on norm}), prove Proposition \ref{prop formula ch} (in \S\ref{ssec proof of prop formula ch}), prove Proposition \ref{prop from folner to exact} (in \S\ref{ssec clms red 1}) and explain Remark \ref{rem ctemp red is temp} (in Claim \ref{clm Prop tempered is c-tempered reductive holds}).

\subsection{}\label{ssec clms red 1}

\begin{lemma}\label{lem red temp is ctemp}
	Let $V$ be an irreducible unitary $G$-representation and suppose that there exists a unit vector $v_0 \in V$ satisfying properties (1) and (2) of Proposition \ref{prop from folner to exact}. Let $0 < r_0 < r_1 < \ldots$ be a sequence such that $\lim_{n \to +\infty} r_n = +\infty$. Then $V$ is c-tempered with F{\o}lner sequence $G_{< r_0} , G_{<r_1}, \ldots$.
\end{lemma}

\begin{proof}
	Property (1) of Definition \ref{def c-temp} is immediate from property (1) of Proposition \ref{prop from folner to exact}. Let us check property (2) of Definition \ref{def c-temp}. Thus, let $v_1 , v_2 \in V$ and let $K \subset G$ be a compact subset. Fix $r^{\prime} \ge 0$ big enough so that $K \subset G_{< r^{\prime}}$ and $K^{-1} \subset G_{< r^{\prime}}$. We then have, for all $r>0$ and all $g_1 , g_2 \in K$: $$G_{< r} \triangle g_2^{-1} G_{< r} g_1 \subset G_{< r + 2 r^{\prime}} \smallsetminus G_{< r - 2 r^{\prime}}.$$ Therefore, using property (2) of Proposition \ref{prop from folner to exact}, $$ \limsup_{r \to +\infty} \frac{\sup_{g_1 , g_2 \in K} \Mss{v_1}{v_2}{G_{< r} \triangle g_2^{-1} G_{< r} g_1}}{ \Mss{v_0}{v_0}{G_{<r}} } \leq \limsup_{r \to +\infty} \frac{ \Mss{v_1}{v_2}{G_{< r + 2 r^{\prime}} \smallsetminus G_{\leq r - 2 r^{\prime} }} }{ \Mss{v_0}{v_0}{G_{<r}}} = 0$$ and therefore also $$  \lim_{n \to +\infty} \frac{\sup_{g_1 , g_2 \in K} \Mss{v_1}{v_2}{G_{< r_n} \triangle g_2^{-1} G_{< r_n} g_1} }{ \Mss{v_0}{v_0}{G_{< r_n}}} = 0.$$
\end{proof}

\begin{proof}[Proof (of Proposition \ref{prop from folner to exact}).]

	Let us fix a $K$-finite unit vector $v'_0 \in V$, for some maximal compact subgroup $K \subset G$. Let $0 < r_0 < r_1 < \ldots$ be a sequence such that $\lim_{n \to +\infty} r_n = +\infty$. By Lemma \ref{lem red temp is ctemp} $V$ is c-tempered with F{\o}lner sequence $G_{<r_0} , G_{<r_1} , \ldots$ and hence by Proposition \ref{prop c-tempered orth rel} we obtain $$ \lim_{n \to +\infty} \frac{\int_{g \in G_{<r_n}} \langle gv_1 , v_2 \rangle \overline{\langle gv_3 , v_4 \rangle} \cdot dg}{\Ms{v'_0}{v'_0}{r_n}} = \langle v_1 , v_3 \rangle \overline{\langle v_2 , v_4 \rangle}$$ for all $v_1 , v_2 , v_3 , v_4 \in V$. Since this holds for any such sequence $\{ r_n \}_{n \ge 0}$, we obtain \begin{equation}\label{eq formula again} \lim_{r \to +\infty} \frac{\int_{g \in G_{<r}} \langle gv_1 , v_2 \rangle \overline{\langle gv_3 , v_4 \rangle} \cdot dg}{\Ms{v'_0}{v'_0}{r}} = \langle v_1 , v_3 \rangle \overline{\langle v_2 , v_4 \rangle} \end{equation} for all $v_1 , v_2 , v_3 , v_4 \in V$. By Theorem \ref{thm main Kfin} we have $$ \lim_{r \to +\infty} \frac{\Ms{v'_0}{v'_0}{r}}{r^{\mathbf{d}(V)}} = C$$ for some $C>0$. This enables to rewrite (\ref{eq formula again}) as $$  \lim_{r \to +\infty} \frac{\int_{g \in G_{<r}} \langle gv_1 , v_2 \rangle \overline{\langle gv_3 , v_4 \rangle} \cdot dg}{r^{\mathbf{d} (V)}} = C \cdot \langle v_1 , v_3 \rangle \overline{\langle v_2 , v_4 \rangle} $$ for all $v_1 , v_2 , v_3 , v_4 \in V$, as desired.

\end{proof}

\subsection{}

\begin{claim}\label{clm doesnt depend on norm}
	The validity of Conjecture \ref{main conj red exact}, as well as the resulting invariants $\mathbf{d} (V)$ and $\mathbf{f} (V)$, of Theorem \ref{thm main Kfin} as well as the resulting invariants $\mathbf{d} (V)$ and $\mathbf{f} (V)$ (the latter in the non-Archimedean case), and of Proposition \ref{prop from folner to exact}, do not depend on the choice of the norm $|| - ||$ on $\frakg$.
\end{claim}

\begin{proof}

Let $|| - ||^{\prime}$ be another norm on $\frakg$, let $\mathbf{r}^{\prime} : G \to \bbR_{\ge 0}$ be the resulting function, and let $G_{<r}^{\prime} \subset G$ be the resulting subsets. There exists $r_0 \ge 0$ such that $$ e^{-r_0} \cdot || X ||  \leq ||X ||^{\prime} \leq e^{r_0} \cdot || X||, \quad \forall X \in \frakg$$ and therefore $$ e^{-2r_0} \cdot ||{\rm Ad} (g)|| \leq || {\rm Ad} (g) ||' \leq e^{2 r_0} \cdot || {\rm Ad} (g) ||, \quad \forall g \in G.$$ Then $$ G^{\prime}_{< r} \subset G_{<r + 2r_0}, \quad \forall r\ge 0$$ and $$ G_{< r} \subset G^{\prime}_{<r + 2r_0}, \quad \forall r\ge 0.$$ These ``sandwich" relations readily imply the independence claims.

\end{proof}

\subsection{}

\begin{claim}\label{clm Prop tempered is c-tempered reductive holds}
	An irreducible unitary $G$-representation for which there exists a unit vector $v_0 \in V$ such that conditions (1) and (2) of Proposition \ref{prop from folner to exact} are satisfied is tempered.
\end{claim}

\begin{proof}
	Clear from Lemma \ref{lem red temp is ctemp} coupled with Corollary \ref{cor c-temp are temp}.
\end{proof}

\subsection{}

\begin{claim}\label{clm counterexample}
	Let $G := PGL_2 (\Omega)$, $\Omega$ a local field. Let $A \subset G$ be the subgroup of diagonal matrices. Then, for every non-trivial irreducible unitary $G$-representation $V$, the set of matrix coefficients of $V$ restricted to $A$ is equal to the set of function on $A$ of the form $$ a \mapsto \int_{\hat{A}} \chi (a) \cdot \phi (\chi) \cdot d\chi$$ as $\phi$ runs over $ L^1 (\hat{A})$.
\end{claim}

\begin{proof}

Denote by $B \subset G$ the subgroup of upper-triangular matrices and by $N \subset B$ its unipotent radical.

\medskip

Let us recall that, by Mackey theory, there is a unique (up to isomorphism) infinite-dimensional irreducible unitary $B$-representation $W$, and the rest of irreducible unitary $B$-representations are killed by $N$. The restriction ${\rm Res}^B_A W$ is isomorphic to the right regular unitary $A$-representation $L^2 (A)$.

\medskip

Let now $V$ be a non-trivial irreducible unitary $G$-representation. Recall that by the Howe-Moore theorem (or by a step in one of its usual proofs) $V$ does not contain non-zero $N$-invariant vectors. By decomposing the restriction ${\rm Res}^G_B V$ into a direct integral of irreducible unitary $B$-representations, and using the fact that $V$ admits no non-zero $N$-invariant vectors, we see that ${\rm Res}^G_B V$ is a multiple of $W$. Hence, we deduce that ${\rm Res}^G_A A$ is a multiple of the right regular unitary $A$-representation $L^2 (A)$.

\medskip

Now, the matrix coefficients of a multiple of the right regular unitary $A$-representation $L^2 (A)$ are easily seen to be the functions on $A$ of the form $$ a \mapsto \int_{\hat{A}} \chi (a) \cdot \phi (\chi) \cdot d\chi$$ where $\phi \in L^1 (\hat{A})$.

\end{proof}

\subsection{}\label{ssec proof of prop formula ch}

\begin{proof}[Proof (of Proposition \ref{prop formula ch}).]
	Fix $d \in D_c^{\infty} (G)$. Let $K \subset G$ be an open compact subgroup such that $d$ is invariant under $K$ both on left and on right. Let us denote by $e_1 , \ldots , e_n$ an orthonormal basis of $V^K$, and let us denote by $\pi_K : V \to V^K$ the orthonormal projection. Let us denote by $[-,-] : C^{-\infty} (G) \times D_c^{\infty} (G) \to \bbC$ the canonical pairing. We have
$$
[{}^g m_{v_1 , v_2} , d] = [m_{g v_1 , g v_2} , d] = \langle d g v_1 , g v_2 \rangle = \langle d \pi_K (g v_1) , \pi_K (g v_2) \rangle = $$ $$ = \sum_{1 \leq i,j \leq n} \langle g v_1 , e_i \rangle \overline{\langle g v_2 , e_j \rangle} \langle d e_i , e_j \rangle.
$$

Hence $$ \frac{\int_{G_{<r}} [{}^g m_{v_1 , v_2} , d] \cdot dg}{r^{\mathbf{d} (V)}} = \sum_{1 \leq i,j \leq n} \langle d e_i , e_j \rangle \cdot \frac{\int_{G_{<r}} \langle g v_1 , e_i \rangle \overline{\langle g v_2 , e_j \rangle} \cdot dg}{r^{\mathbf{d} (V)}}$$ and therefore $$ \lim_{r \to +\infty}  \frac{\int_{G_{<r}} [{}^g m_{v_1 , v_2} , d] \cdot dg}{r^{\mathbf{d} (V)}} = \frac{1}{\mathbf{f} (V)} \sum_{1 \leq i,j \leq n} \langle d e_i , e_j \rangle \cdot \langle v_1 , v_2 \rangle \overline{\langle e_i , e_j \rangle} = $$ $$ = \frac{1}{\mathbf{f} (V)} \sum_{1 \leq i \leq n} \langle d e_i , e_i \rangle \cdot \langle v_1 , v_2 \rangle = \frac{\langle v_1 , v_2 \rangle}{\mathbf{f} (V)} \Theta_V (d).$$
\end{proof}

\section{The case of the principal series representation $V_1$ of slowest decrease}\label{sec proof of V1}

In this section $G$ is a semisimple group over a local field. We continue with notations from \S\ref{sec intro}. Our goal is to prove Theorem \ref{thm intro slowest} (restated as Theorem \ref{thm V1 c temp} below).

\subsection{}\label{ssec V1 and Xi}

We fix a minimal parabolic $P \subset G$ and a maximal compact subgroup $K \subset G$ such that $G = PK$. We consider the principal series unitary $G$-representation $V_1$ consisting of functions $f : G \to \bbC$ satisfying $$ f(pg) = \Delta_P (p)^{1/2} \cdot f(g) \quad \forall p \in P , g \in G$$ where $\Delta_P : P \to \bbR^{\times}_{>0}$ is the modulus function of $P$. The $G$-invariant inner product on $V_1$ can be taken to be $$ \langle f_1 , f_2 \rangle = \int_K f_1 (k) \cdot \overline{f_2 (k)} \cdot dk$$ (where we normalize the Haar measure on $K$ to have total mass $1$). Recall that $V_1$ is irreducible. We denote by $f_0 \in V_1$ the spherical vector, determined by $f_0 (k) = 1$ for all $k \in K$. We also write $$ \Xi_G (g) := \langle g f_0 , f_0 \rangle.$$

\begin{lemma}\label{lem growth Xi}
	Given $r' \ge 0$ we have $$ \lim_{r \to +\infty} \frac{\int_{G_{<r+r'} \smallsetminus G_{<r-r'}} \Xi_G (g)^2 \cdot dg}{\int_{G_{<r}} \Xi_G (g)^2 \cdot dg}  = 0$$ and $$ \lim_{r \to +\infty} \frac{\int_{G_{<r+r'}} \Xi_G (g)^2 \cdot dg}{\int_{G_{<r}} \Xi_G (g)^2 \cdot dg} = 1.$$
\end{lemma}

\begin{proof}
	The second equality follows from the first, and the first is immediately implied by Theorem \ref{thm main Kfin}.
\end{proof}

\subsection{} The main result of this section is:

\begin{theorem}\label{thm V1 c temp}
	Let $V$ be an irreducible tempered unitary $G$-representation. Suppose that there exist a unit vector $v_0 \in V$ such that \begin{equation}\label{eq good vector}
\underset{r \to +\infty}{\limsup} \frac{ \int_{G_{<r}} \Xi_G (g)^2 \cdot dg }{ \Ms{v_0}{v_0}{r} } < +\infty.
\end{equation} Then Conjecture \ref{main conj red exact} holds for $V$. In particular, Conjecture \ref{main conj red exact} holds for $V_1$.
\end{theorem}

\subsection{}

We will prove Theorem \ref{thm V1 c temp} using the following result:

\begin{claim}\label{clm est using Xi}
	Let $V$ be a tempered unitary $G$-representation. Then for all unit vectors $v_1,v_2 \in V$ and all measurable $K$-biinvariant subsets $S \subset G$ we have $$ \int_S |\langle gv_1 , v_2 \rangle |^2 \cdot dg \leq \int_S \Xi_G (g)^2 \cdot dg.$$
\end{claim}

\begin{proof}[Proof (of Theorem \ref{thm V1 c temp} given Claim \ref{clm est using Xi}).]
	
	To show that Conjecture \ref{main conj red exact} holds for $V$ we will use Proposition \ref{prop from folner to exact}, applied to our $V$ and our $v_0$.
	
	\medskip
	
	There exists $r_0 \ge 0$ such that $K G_{<r} K \subset G_{< r + r_0}$ for all $r \ge 0$.
	
	\medskip
	
	Let us verify condition (1) of Proposition \ref{prop from folner to exact}. For unit vectors $v_1 , v_2 \in V$ we have $$ \frac{\Ms{v_1}{v_2}{r}}{\Ms{v_0}{v_0}{r}} \leq \frac{\int_{G_{<r+r_0}} \Xi_G (g)^2 \cdot dg}{\Ms{v_0}{v_0}{r}}$$ and therefore condition (1) of Proposition \ref{prop from folner to exact} follows from (\ref{eq good vector}) and Lemma \ref{lem growth Xi}. Let us now verify condition (2) of Proposition \ref{prop from folner to exact}. For unit vectors $v_1 , v_2 \in V$ and $r' \ge 0$ we have $$ \frac{ \Ms{v_1}{v_2}{r+r'} - \Ms{v_1}{v_2}{r-r'}}{\Ms{v_0}{v_0}{r}} \leq \frac{\int_{G_{<r+r'+r_0} \smallsetminus G_{<r-(r'+r_0)}} \Xi _G(g)^2 \cdot dg}{\int_{G_{<r}} \Xi_G (g)^2 \cdot dg} \cdot \frac{\int_{G_{<r}} \Xi_G (g)^2 \cdot dg}{\Ms{v_0}{v_0}{r}}$$ and therefore condition (2) of Proposition \ref{prop from folner to exact} follows from (\ref{eq good vector}) and Lemma \ref{lem growth Xi}.
\end{proof}

\subsection{}

We will prove Claim \ref{clm est using Xi} using the following result:

\begin{claim}\label{clm norm of conv}
Let $\phi \in L^2 (G)$ be zero outside of a measurable $K$-biinvariant subset $S \subset G$ of finite volume. Denote by $T_{\phi} : L^2 (G) \to L^2 (G)$ the operator of convolution $\psi \mapsto \phi \star \psi$. Then\footnote{Here $|| \phi ||$ stands for the $L^2$-norm of $\phi$.} $$ || T_{\phi} ||^2 \leq \left( \int_{S} \Xi_G (g)^2 \cdot dg \right) \cdot || \phi ||^2.$$
\end{claim}

\begin{proof}[Proof (of Claim \ref{clm est using Xi} given Claim \ref{clm norm of conv}).]
We can clearly assume that $S$ has finite volume. Let us denote $$ \phi (g) := {\rm ch}_{S} (g) \cdot \overline{\langle gv_1 , v_2 \rangle},$$ where ${\rm ch}_S$ stands for the characteristic function of $S$. Let us denote by $S_{\phi} : V \to V$ the operator $$ v \mapsto \int_G \phi (g) \cdot gv \cdot dg.$$ Since $V$ is tempered, we have $ ||S_{\phi}|| \leq || T_{\phi}||$. Therefore $$ \int_{S} |\langle gv_1 , v_2 \rangle |^2 \cdot dg = \int_G \phi (g) \cdot \langle gv_1 , v_2 \rangle \cdot dg = \langle S_{\phi} v_1 , v_2 \rangle \leq || S_{\phi} || \leq || T_{\phi} || \leq $$ $$ \leq \left( \sqrt{\int_S \Xi_G (g)^2 \cdot dg} \right) \cdot || \phi || = \left( \sqrt{\int_S \Xi_G (g)^2 \cdot dg} \right) \cdot  \left( \sqrt{\int_S |\langle gv_1 , v_2 \rangle |^2 \cdot dg} \right)$$ thus $$\int_S |\langle gv_1 , v_2 \rangle |^2 \cdot dg \leq \int_S \Xi_G (g)^2 \cdot dg$$ as desired.
\end{proof}

\subsection{}

Finally, let us prove Claim \ref{clm norm of conv}, following \cite{ChPiSa}.

\begin{proof}[Proof (of Claim \ref{clm norm of conv}).]
By\footnote{In the lemma we refer to it is assumed that $\phi$ is continuous but the arguments there apply to our $\phi$ without any modification.} \cite[Lemma 3.5]{ChPiSa} we can assume that $\phi$ is $K$-biinvariant and non-negative. By \cite[Proposition 4.3]{ChPiSa} we have $$ || T_{\phi} || = \int_G \Xi_G (g) \cdot \phi (g) \cdot dg.$$ Applying the Cauchy-Schwartz inequality, we obtain $$ ||T_{\phi} ||^2 \leq \left( \int_S \Xi_G (g)^2 \cdot dg\right) \cdot|| \phi||^2,$$ as desired.
\end{proof}

\section{The proof of Theorem \ref{thm intro SL2R}}\label{sec proof of SL2R}

In this section we let $G$ be either $SL_2 (\bbR)$ or $PGL_2 (\Omega)$, where $\Omega$ is a non-Archimedean local field of characteristic $0$ and residual characteristic not equal to $2$. We prove Theorem \ref{thm intro SL2R}.

\subsection{}

If $G = PGL_2 (\Omega)$, we denote by $\varpi$ a uniformizer in $\Omega$, by $\calO$ the ring of integers in $\Omega$, by $p$ the residual characteristic of $\Omega$ and $q:= |\calO / \varpi \calO |$.

\subsection{}\label{ssec SL2 notation}

We denote by $A \subset G$ the subgroup of diagonal matrices and by $U \subset G$ the subgroup of unipotent upper-triangular matrices. If $G = SL_2 (\bbR)$ we define the isomorphism $$ \mathbf{a} : \bbR^{\times} \to A, \quad t \mapsto \mtrx{t}{0}{0}{t^{-1}}$$ and if $G = PGL_2 (\Omega)$ we define the isomorphism $$ \mathbf{a} : \Omega^{\times} \to A, \quad  t \mapsto \mtrx{t}{0}{0}{1}.$$ We denote $A^+ := \{ a \in A \ | \ |\mathbf{a}^{-1} (a) | \ge 1 \}$.

\medskip

If $G = SL_2 (\bbR)$ then we can (and will) take $||- ||$ on $\frakg$ to be such that $$\mathbf{r} \left( k_1 \mtrx{t}{0}{0}{t^{-1}} k_2 \right) = \log \max \{ |t|^2 , |t|^{-2} \}$$ where $t\in \bbR^{\times}$ and $k_1 , k_2 \in SO (2)$. If $G = PGL_2 (\Omega)$ then we can (and will) take $|| - ||$ on $\frakg$ to be such that $$\mathbf{r} \left( k_1 \mtrx{t}{0}{0}{s} k_2 \right) = \log \max \{ |t/s| , |s/t| \}$$ where $t,s \in \Omega^{\times}$ and $k_1 , k_2 \in PGL_2 (\calO)$. Let us denote $A^+_{<r} := A^+ \cap G_{<r}$.

\medskip

If $G = SL_2 (\bbR)$ we set $K := SO (2) \subset G$. If $G = PGL_2 (\Omega)$ we choose a non-square $\zeta \in \calO^{\times}$ and set $K \subset G$ to be the subgroup of elements of the form $\mtrx{a}{\zeta b}{b}{a}$, $(a,b) \in \Omega^2 \smallsetminus \{ (0,0) \}$ (so $K$ is a closed compact subgroup in $G$, but not open, and in particular not maximal).

\medskip

We set $\omega : A^+ \to \bbR_{\ge 0}$ to be given by $\omega (\mathbf{a} (t)) := |t^2 - t^{-2} |$ if $G = SL_2 (\bbR)$ and $\omega (\mathbf{a} (t)) := |t - t^{-1} |$ if $G = PGL_2 (\Omega)$. Then, taking the Haar measure on $K$ to have total mass $1$ and appropriately normalizing the Haar measure on $A$, for all non-negative-valued measurable functions $f$ on $G$ we have $$ \int_G f(g) \cdot dg = \int_{A^+} \omega (a) \left( \int_{K \times K} f(k_2 a k_1) \cdot dk_1 dk_2 \right) da.$$

\medskip

Given a unitary $G$-representation $V$, vectors $v_1 , v_2 \in V$ and $a \in A^+$, we write $$ \ms{v_1}{v_2}{a} := \int_{K \times K} | \langle k_2 a k_1 v_1 , v_2 \rangle |^2 \cdot dk_1 dk_2.$$ We have \begin{equation}\label{eq M in term of Mcirc} \Ms{v_1}{v_2}{r} := \int_{A^+_{<r}} \omega (a) \cdot \ms{v_1}{v_2}{a} \cdot da \end{equation} (where $M_{v_1 , v_2} (r)$ was already defined in \S\ref{sec intro}).

\medskip

Given a unitary character\footnote{${\rm U} (1)$ denotes the subgroup of $\bbC^{\times}$ consisting of complex numbers with absolute value $1$.} $\chi: A \to {\rm U}(1) $ we consider the principal series unitary $G$-representation $V_{\chi}$, consisting of functions $f : G \to \bbC$ satisfying $$ f(u a g) = \chi (a) \cdot \Delta (a)^{1/2} \cdot f(g) \quad \forall a \in A, u \in U , g \in G,$$ where $\Delta (a) = |\mathbf{a}^{-1} (a)|^2$ if $G = SL_2 (\bbR)$ and $\Delta (a) = |\mathbf{a}^{-1} (a)|$ if $G = PGL_2 (\Omega)$. Here $G$ acts by $(g^{\prime}f)(g) := f(gg^{\prime})$. The $G$-invariant inner product on $V_{\chi}$ can be expressed as $$ \langle f_1 , f_2 \rangle = \int_{K} f_1 (k) \cdot \overline{f_2 (k)} \cdot dk.$$ For $\theta \in \hat{K}$, let $h^{\chi}_{\theta} \in V_{\chi}$ denote the unique vector determined by $h^{\chi}_{\theta} (k) = \theta (k)$ for $k \in K$, if it exists, and write $\textnormal{types} (V_{\chi}) \subset \hat{K}$ for the subset of $\theta$'s for which it exists. Thus $(h^{\chi}_{\theta})_{\theta \in \textnormal{types} (V_{\chi})}$ is a Hilbert basis for $V_{\chi}$.

\subsection{}

Let us now give several preparatory remarks.

\medskip

First, we do not establish Conjecture \ref{main conj red exact} directly but, rather, establish conditions (1) and (2) of Proposition \ref{prop from folner to exact} (which suffices by that proposition).

\medskip

Second, for a square-integrable irreducible unitary $G$-representation $V$, establishing conditions (1) and (2) of Proposition \ref{prop from folner to exact} with any unit vector $v_0 \in V$ is straight-forward (see the proof of Proposition \ref{prop sq int} for a spelling-out). As is well-known, a tempered irreducible unitary $G$-representation which is not square-integrable is a direct summand in some $V_{\chi}$. Therefore, we establish conditions (1) and (2) of Proposition \ref{prop from folner to exact} for irreducible direct summands in $V_{\chi}$.

\medskip

Third, if $\chi = 1$ when $G = SL_2 (\bbR)$ or if $\chi^2 = 1$ when $G = PGL_2 (\Omega)$, $V_{\chi}$ satisfies Conjecture \ref{main conj red exact} by Theorem \ref{thm V1 c temp}. So we assume throughout:

\begin{equation}\label{eq condition} \chi \neq 1 \textnormal{ if } G=SL_2 (\bbR), \quad \quad \quad \chi^2 \neq 1 \textnormal{ if } G = PGL_2 (\Omega). \end{equation}

\subsection{}

We reduce Conjecture \ref{main conj red exact} for an irreducible summand in $V_{\chi}$ to the following two claims.

\begin{claim}\label{clm SL2 2}
Fix $\chi$ satisfying (\ref{eq condition}). Let $V$ be an irreducible direct summand in $V_{\chi}$. There exist $f \in V$, $r_0 \ge 0$ and $D>0$ such that for all $r \ge r_0$ we have \begin{equation}\label{eq from below} \Ms{f}{f}{r} \geq D \cdot r \end{equation}
\end{claim}

\begin{claim}\label{clm SL2 2b}
Fix $\chi$ satisfying (\ref{eq condition}). There exist $r_0 > 0$ and $C>0$ (depending on $\chi$) such that for all $a \in A^+ \smallsetminus A^+_{<r_0}$ we have \begin{equation}\label{eq conj SL2 2} \ms{f_1}{f_2}{a} \leq C \cdot \omega (a)^{-1} \cdot ||f_1||^2 \cdot ||f_2||^2 \quad \quad \forall f_1, f_2 \in V_{\chi}.\end{equation}
\end{claim}

\begin{proof}[Proof (of Conjecture \ref{main conj red exact}  for summands in $V_{\chi}$ given Claim \ref{clm SL2 2} and Claim \ref{clm SL2 2b} for $\chi$).]
	
	Let $V$ be an irreducible direct summand in $V_{\chi}$. Let $f$, $r_0$, $D$ and $C$ be as in Claim \ref{clm SL2 2} and as in Claim \ref{clm SL2 2b} (taking $r_0$ to be the maximum of the values from the two statements).
	
	\medskip
	
	In order to verify Conjecture \ref{main conj red exact} for $V$, we will verify the conditions (1) and (2) of Proposition \ref{prop from folner to exact}, where for $v_0$ we take our $f$.
	
	\medskip
	
	Using (\ref{eq conj SL2 2}) we obtain the existence of $E,E'>0$ such that for all $r_0 \leq r_1 < r_2$ we have $$ \Ms{f_1}{f_2}{r_2} - \Ms{f_1}{f_2}{r_1} \leq E \cdot \textnormal{vol}_A ( A^+_{<r_2} \smallsetminus A^+_{<r_1} )  \cdot || f_1||^2 \cdot || f_2 ||^2 \leq $$ $$ \leq E' \cdot (1 + (r_2 - r_1)) \cdot || f_1||^2 \cdot || f_2 ||^2.$$ From this and (\ref{eq from below}) the conditions (1) and (2) of Proposition \ref{prop from folner to exact} are immediate.
	
\end{proof}

\subsection{}\label{sec proof of clm SL2 2}

Let us prove Claim \ref{clm SL2 2}.

\begin{proof}[Proof (of Claim \ref{clm SL2 2}).]

	Let $V$ be an irreducible direct summand of $V_{\chi}$.
	
	\medskip
	
	Let us first treat the case $G = PGL_2 (\Omega)$. We use the (normalized) Jacquet $A$-module $J(-)$ with respect to $G \hookleftarrow AU \twoheadrightarrow A$. We denote by $\underline{V} \subset V$ the subspace of smooth vectors. $J(\underline{V})$ is isomorphic to $\bbC_{\chi} \oplus \bbC_{\chi^{-1}}$. We consider $v \in \underline{V}$ whose projection under the canonical $\underline{V} \twoheadrightarrow J(\underline{V})$ is non-zero and is an $A$-eigenvector with eigencharacter $\chi$. By Casselman's canonical pairing theory there exists a non-zero $\alpha \in J(\underline{V})^*$ which is $A$-eigenvector with eigencharacter $\chi^{-1}$ such that $\langle a v , v \rangle = |a |^{-1/2} \alpha (a v)$ whenever $a \in A^+ \smallsetminus A^+_{<r_0}$, for large enough $r_0 \ge 0$. Since we have $\alpha (a v) = \chi (a) \cdot \alpha (v)$ and $\alpha (v) \neq 0$, we deduce that for some $C>0$ we have $|\langle a v , v \rangle|^2 = C \cdot |a|^{-1} $ for $a \in A^+ \smallsetminus A^+_{<r_0}$. Let $K_v \subset K$ be an open compact subgroup, small enough so that $K_v v = v$. We have, again for $a \in A^+ \smallsetminus A^+_{< r_0}$: $$ \ms{v}{v}{a} = \int_{K \times K} |\langle k_2 a k_1 v , v \rangle |^2 \cdot dk_1 dk_2 \ge $$ $$ \ge \int_{K_v \times K_v} |\langle k_2 a k_1 v , v \rangle |^2 \cdot dk_1 dk_2 = C^{\prime} \cdot |\langle a v , v \rangle |^2$$ for some $C^{\prime}>0$ and so $\ms{v}{v}{a} \ge C^{\prime \prime} \cdot |a|^{-1}$ for some $C^{\prime \prime}>0$. From this we obtain the desired.
	
	\medskip
	
	Let us now treat the case $G = SL_2 (\bbR)$. Fix any $\theta \in {\rm types} (V_{\chi})$. The leading asymptotic of $K$-finite vectors are well-known, and can be computed from explicit expressions in terms of the hypergeometric function (see \cite[\S 6.5]{KlVi}). In the case $\chi|_{\mathbf{a} (\bbR^{\times}_{>0})} \neq 1$, denoting by $0 \neq s \in \bbR$ the number for which $\chi (\mathbf{a} (t)) = t^{is}$ for all $t \in \bbR^{\times}_{>0}$, we have $$ \langle \mathbf{a} (e^x) h^{\chi}_{\theta} , h^{\chi}_{\theta} \rangle \sim e^{-x} \cdot (E_1 \cdot e^{-isx} + E_2 \cdot e^{isx} + o(1)) \quad\quad (x \to +\infty)$$ for some non-zero $E_1$ and $E_2$ and so $$ |\langle \mathbf{a} (e^x) h^{\chi}_{\theta} , h^{\chi}_{\theta} \rangle|^2 \sim e^{-2x} \left( D + E_3 \cdot e^{-2isx} + E_4 \cdot e^{2isx} + o(1) \right) \quad\quad (x \to +\infty)$$ for some $D>0$, $E_3$ and $E_4$. From this we obtain the desired. In the case $\chi|_{\mathbf{a} (\bbR^{\times}_{>0})} = 1$, and so $\chi (\mathbf{a} (-1)) = - 1$, we have $$ \langle \mathbf{a} (e^x) h_{\theta}^{\chi} , h_{\theta}^{\chi} \rangle \sim E \cdot e^{-x}  \quad\quad (x \to +\infty) $$ for some non-zero $E$ and so $$|  \langle \mathbf{a} (e^x) h_{\theta}^{\chi} , h_{\theta}^{\chi} \rangle |^2 \sim D \cdot e^{-2x}  \quad\quad (x \to +\infty) $$ for some $D>0$. From this we obtain the desired.
	
	\end{proof}

\subsection{}

We further reduce Claim \ref{clm SL2 2b}.

\begin{claim}\label{clm SL2 3}
	Fix $\chi$ satisfying (\ref{eq condition}). There exist $r_0 > 0$ and $C>0$ (depending on $\chi$) such that for all $\theta,\eta \in \textnormal{types} (V_{\chi})$ and all $a \in A^+ \smallsetminus A^+_{<r_0}$ we have \begin{equation}\label{eq thm SL2R 3} |\langle a h^{\chi}_{\theta} , h^{\chi}_{\eta} \rangle|^2 \leq C \cdot \omega(a)^{-1} .\end{equation}
\end{claim}

\begin{proof}[Proof (of Claim \ref{clm SL2 2b} for $\chi$ given Claim \ref{clm SL2 3} for $\chi$).]

	Let $f_1 , f_2 \in V_{\chi}$ and write $$ f_1 = \sum_{\theta \in \textnormal{types} (V_{\chi})} c_{\theta} \cdot h^{\chi}_{\theta}, \quad f_2 = \sum_{\theta \in \textnormal{types} (V_{\chi})} d_{\theta} \cdot h^{\chi}_{\theta}$$ with $c_{\theta} , d_{\theta} \in \bbC$. Using Fourier expansion of the function $(k_1 , k_2) \mapsto \langle a k_1 f_1 , k_2 f_2 \rangle$ on $K \times K$ we have, for $a \in A^+ \smallsetminus A^+_{<r_0}$: $$ \ms{f_1}{f_2}{a} = \int_{K\times K} \left| \langle a k_1 f_1 , k_2 f_2 \rangle \right|^2 \cdot dk_1 dk_2 = $$ $$ = \sum_{\theta,\eta \in \textnormal{types} (V_{\chi})} |c_{\theta}|^2 \cdot |d_{\eta}|^2 \cdot \left| \langle a h^{\chi}_{\theta} ,  h^{\chi}_{\eta} \rangle \right|^2 \leq C  \cdot \omega (a)^{-1} \cdot ||f_1||^2 \cdot ||f_2||^2.$$

\end{proof}

\subsection{}

Let us now establish Claim \ref{clm SL2 3} in the case $G = SL_2 (\bbR)$:

\begin{proof}[Proof (of Claim \ref{clm SL2 3} in the case $G = SL_2 (\bbR)$).]
	In the case $\chi |_{\mathbf{a} (\bbR^{\times}_{>0})} \neq 1$, this is the contents of \cite[Theorem 2.1]{BrCoNiTa} (which contains a stronger claim, incorporating $\chi$ into the inequality). Let us therefore assume $\chi|_{\mathbf{a} (\bbR^{\times}_{>0})} = 1$ and so $\chi (\mathbf{a} (-1)) = -1$. For $n \in \bbZ$, let us denote by $\theta_n$ the character of $K$ given by $\mtrx{c}{-s}{s}{c} \mapsto (c+is)^n$. We want to see that $$ \cosh (x) \left| \langle \mathbf{a} (e^x) h_{\theta_n}^{\chi} , h_{\theta_m}^{\chi} \rangle \right|$$ is bounded as we vary $x \in [0,+\infty)$ and $m,n \in 1 + 2 \bbZ$. We have $V_{\chi} = V_{\chi}^- \oplus V_{\chi}^+$, where $V_{\chi}^-$ and $V_{\chi}^+$ are irreducible unitary $G$-representations, ${\rm types} (V_{\chi}^-) = \{ \theta_n \ : \ n \in -1 - 2 \bbZ_{\ge 0}\}$ and ${\rm types} (V_{\chi}^+) = \{ \theta_n \ : \ n \in 1 + 2 \bbZ_{\ge 0}\}$. Since the matrix coefficient in question vanishes when $m \in -1 - 2 \bbZ_{\ge 0}$ and $n \in 1 + 2 \bbZ_{\ge 0}$ or vice versa, and since the matrix coefficient in question does not change when we replace $(n ,m)$ by $(-n , -m)$, we can assume $m,n \in 1 + 2 \bbZ_{ \ge 0}$. Furthermore, by conjugating by $\mtrx{0}{1}{-1}{0}$ it is straight-forward to see that we can assume that $m \ge n$. We denote $k := \tfrac{m-n}{2}$. Let us use the well-known concrete expression of matrix coefficients in terms of the hypergeometric function (see \cite[\S 6.5]{KlVi}): $$ \cosh (x) \cdot  \langle \mathbf{a} (e^x) h_{\theta_n}^{\chi} , h_{\theta_m}^{\chi} \rangle =  \tanh (x)^k \cdot \frac{(1+\frac{n-1}{2})_k}{k!} \cdot {}_2 F_1 \left( \frac{n-1}{2}+k+1 , \ -\frac{n-1}{2} , \ k+1 , \ \tanh (x)^2 \right).$$ We want to show that this expression is bounded as we vary $x \in [0,+\infty)$, $n \in 1+2 \bbZ_{\ge 0}$ and $k \in \bbZ_{\ge 0}$. Performing a change of variables $\frac{1-t}{2} := \tanh (x)^2$, denoting $r := \tfrac{n-1}{2}$ and interpreting in terms of Jacobi polynomials, we can rewrite this last expression as: $$Q^k_r (t) := (\tfrac{1-t}{2})^{k/2} \cdot P^{(k,0)}_{r} (t).$$ We want to see that $Q^k_r (t)$ is bounded as we vary $t \in [-1,1]$, $r \in \bbZ_{\ge 0}$ and $k \in \bbZ_{\ge 0}$. But, it is known that under a suitable interpretation of the variable $t$, $Q^k_r (t)$ is equal to a matrix coefficient of unit vectors in an irreducible unitary representation of ${\rm SU} (2)$, see, for example, \cite[\S 6.3]{KlVi}. Therefore, we have $ |Q^k_r (t)| \leq 1$ for all $t \in [-1,1]$, $r \in \bbZ_{\ge 0}$ and $k \in \bbZ_{\ge 0}$, as desired. See also \cite[Equation (20)]{HaSc} for a direct proof of this last inequality.
		
\end{proof}

\subsection{}

Finally, we want to establish Claim \ref{clm SL2 3} in the case $G = PGL_2 (\Omega)$. We will use the following proposition:

\begin{proposition}\label{prop padic estimate}
There exists $C>0$ such that the following holds. Let $\psi_1$ and $\psi_2$ be unitary characters of $\varpi \calO$. Let $\alpha_1 , \alpha_2 \in \calO^{\times} x + x ^2\Omega [[x ]] \subset \Omega [[x]]$ be power series. Let $\chi$ be a non-trivial unitary character of $\Omega^{\times}$. Denote by $c(\chi)$ the number $0$ if $\chi|_{\calO^{\times}} = 1$ and otherwise the smallest number $c \in \bbZ_{\ge 1}$ for which $\chi|_{1+\varpi^c \calO} = 1$. Also, denote by $d(\chi)$ the number $1/|1-\chi (\varpi)|$ if $c(\chi) = 0$ and the number $0$ if $c(\chi) \neq 0$. Let $0 < m_1 \leq m_2 \leq n$ be integers. Then $$ \left| \int_{\varpi^{m_1} \calO \smallsetminus \varpi^{m_2} \calO} \psi_1 (\alpha_1 (x)) \psi_2 (\alpha_2 (\varpi^n x^{-1})) \chi (x) \cdot \frac{dx}{|x|} \right| \leq C(c(\chi)+d(\chi) + 1).$$
\end{proposition}

\begin{proof}[Proof (of Claim \ref{clm SL2 3} in the case $G = PGL_2 (\Omega)$ given Proposition \ref{prop padic estimate}).]
Let us calculate more concretely the inner product appearing in Claim \ref{clm SL2 3}. We can normalize the inner product on $V_{\chi}$ so that for $f_1, f_2 \in V_{\chi}$ we have $$ \langle f_1 , f_2 \rangle = \int_{\Omega} f_1 \mtrx{1}{0}{x}{1} \overline{f_2 \mtrx{1}{0}{x}{1}} \cdot dx.$$ We then calculate $$ \left\langle \mtrx{t}{0}{0}{1} h^{\chi}_{\theta} , h^{\chi}_{\mu} \right\rangle = |t|^{-1/2} \chi (t) \int_{\Omega}  h^{\chi}_{\theta} \mtrx{1}{0}{x}{1} \overline{h^{\chi}_{\mu} \mtrx{1}{0}{t^{-1} x}{1}} \cdot dx  $$ and thus we want to see that $$ I_{\theta , \mu} (t) :=  \int_{\Omega}  h^{\chi}_{\theta} \mtrx{1}{0}{x}{1} \overline{h^{\chi}_{\mu} \mtrx{1}{0}{t^{-1} x}{1}} \cdot dx $$ is bounded independently of $\theta , \mu \in \hat{K}$ and $t \in F$ satisfying $|t| \ge 1$. In general, for $\theta \in \hat{K}$ and $x \in \Omega$, we have $$ h^{\chi}_{\theta} \mtrx{1}{0}{x}{1} = \frac{\chi^{-1} (1 - \zeta x^2)}{|1 - \zeta x^2|^{1/2}} \theta \mtrx{1}{\zeta x}{x}{1}.$$ And so we obtain \begin{equation}\label{eq for I} I_{\theta , \mu} (t) = \int_{\Omega} \frac{\chi^{-1} (1 - \zeta x^2)}{|1 - \zeta x^2|^{1/2}}  \frac{\chi (1 - \zeta t^{-2} x^2)}{|1 - \zeta t^{-2} x^2|^{1/2}} \theta \mtrx{1}{\zeta x}{x}{1} \mu^{-1} \mtrx{1}{\zeta t^{-1} x}{t^{-1} x}{1} \cdot dx. \end{equation} If for $D \subset F$ we denote by $I^D_{\theta , \mu} (t)$ the same expression as that for $I_{\theta , \mu} (t)$ in (\ref{eq for I}) but where integration is performed over $D$, we see $$ | I^{\varpi \calO}_{\theta , \mu} (t) | \leq \int_{\varpi \calO} dx = 1/q$$ and thus this is bounded. Furthermore, $$ |I^{\Omega \smallsetminus t \calO}_{\theta , \mu} (t)| \leq \int_{\Omega \smallsetminus t \calO} \frac{dx}{|x| \cdot |t^{-1} x|} = 1/q $$ and thus this is bounded. Finally, for any specific $- \log_q |t| \leq m \leq 0$, we have $$ |I^{\varpi^{m} \calO^{\times}}_{\theta , \mu} (t)| \leq \int_{\varpi^{m} \calO^{\times}}  \frac{dx}{|x|} =1-1/q.$$ Therefore, denoting by $k \in \bbZ_{\ge 1}$ a number such that $\chi|_{1+\varpi^{2k} \calO} = 1$, it is enough to bound $I^{\varpi^k t \calO \smallsetminus \varpi^{-k} \calO}_{\theta , \mu} (t)$. We have $$I^{\varpi^k t \calO \smallsetminus \varpi^{-k} \calO}_{\theta , \mu} (t) = $$ $$ = \chi^{-1} (- \zeta) \theta \mtrx{0}{\zeta}{1}{0} \int_{\varpi^k t \calO \smallsetminus \varpi^{-k} \calO} \chi^{-2} (x) \theta \mtrx{1}{x^{-1}}{\zeta^{-1} x^{-1} }{1} \mu^{-1} \mtrx{1}{\zeta t^{-1} x}{t^{-1} x}{1} \cdot \frac{dx}{|x|}.$$ Let us denote by $K^{\prime} \subset K$ the subgroup consisting of $\mtrx{x}{\zeta y}{y}{x}$ for which $|y| \leq |x| \cdot |p|$. We have an isomorphism of topological groups $$ e : p \calO \to K^{\prime}$$ given by $$ e (y) := \exp \mtrx{0}{\zeta y}{y}{0}.$$ Let us denote by $\alpha : p \calO \to p \calO$ the map given by $\alpha (y) := e^{-1} \mtrx{1}{\zeta y}{y}{1}$. We have a power series expansion $\alpha (y) = y + \zeta y^3 / 3 + \zeta^2 y^5 / 5 + \ldots$. Let us now denote by $\widetilde{\theta}$ the unitary character of $p \calO$ satisfying $\theta|_{K'} \circ e = \widetilde{\theta}$ and by $\widetilde{\mu}$ the unitary character of $p \calO$ satisfying $\mu|_{K'} \circ e = \widetilde{\mu}$. Returning to our integral, we can take $k$ big enough so that $\varpi^k \in p \calO$. Substituting $x^{-1}$ in place of $x$ in the integral we have, we see that we need to show that $$ \int_{\varpi^{k+1} \calO \smallsetminus \varpi^{-k+1} t^{-1} \calO} \widetilde{\theta} (\alpha (\zeta^{-1} x)) \widetilde{\mu}^{-1} (\alpha (t^{-1} x^{-1})) \cdot \chi^2 (x) \cdot \frac{dx}{|x|}$$ is bounded independently of $\widetilde{\theta} , \widetilde{\mu} \in \widehat{p \calO}$ and $t \in F$ satisfying $|t| \ge 1$. This is implied by Proposition \ref{prop padic estimate}.
\end{proof}

\begin{remark}
	We see from the proof that we have a more precise version of Claim \ref{clm SL2 3}: There exists $C>0$ such that for $\chi$ satisfying (\ref{eq condition}), $\theta , \mu \in {\rm types} (V_{\chi})$ and $t \in \Omega^{\times}$ satisfying $|t| \ge 1$, we have $$ \left| \left\langle \mtrx{t}{0}{0}{1} h^{\chi}_{\theta} , h^{\chi}_{\mu} \right\rangle \right| \leq C \cdot |t|^{-1/2} \cdot (c(\chi) +d(\chi)+1).$$ Here $c(\chi)$ and $d(\chi)$ are as in the formulation of Proposition \ref{prop padic estimate} (when we identify $\chi$ with a character of $\Omega^{\times}$ via the isomorphism $\ba : \Omega^{\times} \to A$).
\end{remark}

Let us prove Proposition \ref{prop padic estimate}.

\begin{proof}[Proof (of Proposition \ref{prop padic estimate}).]
	We can assume that the $1$-th coefficients of $\alpha_1$ and $\alpha_2$ are equal to $1$. Let us fix a unitary character $\psi$ of $\Omega$ satisfying $\psi|_{\calO} = 1$ and $\psi|_{\varpi^{-1} \calO} \neq 1$. For $i \in \{ 1 , 2 \}$, let $a_i \in \Omega$ be such that $\psi (a_i x) = \psi_i (x)$ for all $x \in \varpi \calO$ ($a_i$ are defined up to addition of elements in $\varpi^{-1} \calO$, so in particular we can assume that $a_i \neq 0$). Given $0 < m < n$ we define $$ J^m  := \int_{\varpi^m \calO^{\times}} \psi \left( a_1 \alpha_1 (x) + a_2 \alpha_2 (\varpi^n x^{-1}) \right) \cdot \chi (x) \cdot \frac{dx}{|x|} = $$ $$ = \chi (\varpi)^{m} \int_{ \calO^{\times} } \psi \left( a_1 \alpha_1 (\varpi^m x) + a_2 \alpha_2 (\varpi^{n-m} x^{-1}) \right) \cdot \chi (x) \cdot dx,$$ so that the integral in question equals $\sum_{m_1 \leq m < m_2} J^m$. Let us abbreviate $b := (a_2 \varpi^{n-m}) / (a_1 \varpi^m)$. As we vary $m$, let us divide into cases.

\begin{enumerate}

	\item Assume that $ |\varpi^m | < q^{-c (\chi)}$ and that $|b| < q^{-c(\chi)}$. Set $$\beta (x) := \varpi^{-m} \alpha_1 (\varpi^m x) + b \cdot  \varpi^{m-n} \alpha_2 (\varpi^{n-m} x^{-1}).$$ Then $\beta$ gives a well-defined invertible analytic map $\calO^{\times} \to \calO^{\times}$, whose derivative is everywhere a unit. Moreover, if ${c(\chi)} > 0$, we have $\beta^{-1} (x_0 (1 + \varpi^{c(\chi)} \calO)) = x_0 (1 + \varpi^{c(\chi)} \calO)$ for any $x_0 \in \calO^{\times}$. Hence $$ J^m  = \chi (\varpi)^{m} \int_{ \calO^{\times}} \psi (a_1 \varpi^m \beta (x)) \cdot \chi (x) \cdot dx = $$ $$ = \chi (\varpi)^{m} \int_{ \calO^{\times} } \psi (a_1 \varpi^m x) \cdot \chi (x) \cdot dx.$$ Therefore:

\begin{enumerate}
	\item Suppose that $|a_1 \varpi^m| \leq 1$. Then $J^m = \chi (\varpi)^{m} (1-1/q)$ if $\chi$ is unramified and $J^m = 0$ if $\chi$ is ramified.
	\item Suppose that $|a_1 \varpi^m \varpi^{c(\chi)} | > q^{-1}$. Then $J^m = 0$ if $\chi$ is unramified. If $\chi$ is ramified, we write $$ J^m  = \chi (\varpi)^{m} \sum_{x_0 \in \calO^{\times} / (1+\varpi^{c(\chi)} \calO)} \psi (a_1 \varpi^m x_0) \chi (x_0) \int_{\varpi^{c(\chi)} \calO} \psi (a_1 \varpi^m x) \cdot dx$$ and each integral here is equal to $0$, so that also in that case we obtain $J^m = 0$.
	\item The case when neither of these two cases is satisfied corresponds to only finitely many values of $m$, whose number is linearly bounded in terms of $c(\chi)+1$, and so we can be content with the crude estimate $|J^m| \leq 1$ in this case.
\end{enumerate}

\item Assume that $|\varpi^{n-m}| < q^{-c(\chi)}$ and $|b^{-1}| < q^{-c ( \chi)}$. This case is dealt with analogously to the previous one; one denotes $$ \beta (x) := \omega^{m-n} \alpha_2 (\varpi^{n-m} x^{-1}) + b^{-1} \varpi^{-m} \alpha_1 (\varpi^m x)$$ and gets $$ J^m = \chi (\varpi)^m \int_{\calO^{\times}} \psi (a_2 \varpi^{n-m} x) \cdot \chi (x) \cdot dx.$$ And thus:

\begin{enumerate}
	\item Suppose that $|a_2 \varpi^{n-m}| \leq 1$. Then $J^m = \chi (\varpi)^{m} (1-1/q)$ if $\chi$ is unramified and $J^m = 0$ if $\chi$ is ramified.
	\item Suppose that $|a_2 \varpi^{n-m} \varpi^{c(\chi)} | > q^{-1}$. Then $J^m = 0$ if $\chi$ is unramified. If $\chi$ is ramified, we write $$ J^m  = \chi (\varpi)^{m} \sum_{x_0 \in \calO^{\times} / (1+\varpi^{c(\chi)} \calO)} \psi (a_2 \varpi^{n-m} x_0) \chi (x_0) \int_{\varpi^{c(\chi)} \calO} \psi (a_2 \varpi^{n-m} x) \cdot dx$$ and each integral here is equal to $0$, so that also in that case we obtain $J^m = 0$.
	\item The case when neither of these two cases is satisfied corresponds to only finitely many values of $m$, whose number is linearly bounded in terms of $c(\chi)+1$, and so we can be content with the crude estimate $|J^m| \leq 1$ in this case.
\end{enumerate}

	\item The case when neither of these two cases is satisfied corresponds to only finitely many values of $m$, whose number is linearly bounded in terms of $c(\chi)$, and so we can be content with the crude estimate $|J^m| \leq 1$ in this case.
	
\end{enumerate}

As our integral in question is equal to $\sum_{m_1 \leq m < m_2} J^m$, it is straight-forward that the findings above give the boundedness as desired.
\end{proof}

\appendix

\section{Auxiliary claims regarding polynomial growth of exponential integrals and sums}

\subsection{Some notation}

We denote $[n] := \{1 , 2, \ldots , n \}$. We denote $$\bbC_{\leq 0} := \{ z \in \bbC \ | \ {\rm Re} (z) \leq 0\}, \quad D := \{ z \in \bbC \ | \ |z| \leq 1\}.$$ Given $x = (x_1 , \ldots , x_n) \in \bbR_{\ge 0}^n$ and $m = (m_1 , \ldots , m_n) \in \bbZ_{\ge 0}^n$, we write $x^m := x_1^{m_1} \ldots x_n^{m_n}$. Given $\lambda \in \bbC_{\leq 0}^n$ we denote $$J_{\lambda} := \{ 1 \leq j \leq n \ | \ {\rm Re} (\lambda_j) = 0\}.$$ Given $(\lambda , m) \in \bbC_{\leq 0}^n \times \bbZ_{\ge 0}^n$, we denote $d(\lambda , m) := \sum_{j \in J_{\lambda}} (1+m_j)$. Given $J \subset [n]$ and some set $X$, let us denote by ${\rm res}_{J} : X^n \to X^{J}$ the natural restriction and by ${\rm ext}^{J} : X^{J} \to X^{n}$ the natural extension by zero.

\medskip

We fix a finite set $\calI \subset \bbR_{\ge 0}^n$ with the property that given $j \in [n]$ there exists $v \in \calI$ such that $\langle v , e_j \rangle \neq 0$, where $e_j$ the $j$-th standard basis vector. We denote $$ P_{< r} := \{ x \in \bbR_{\ge 0}^n \ | \ \langle v , x \rangle < r \ \forall v \in \calI \}.$$ Given $J \subset [n]$, we denote by $P_J \subset \bbR_{\ge 0}^J$ the convex pre-compact subset $\{ y \in \bbR_{\ge 0}^J \ | \ {\rm ext}^{J} (y) \in P_{<1}\}$.

\medskip

In \S\ref{ssec growth integral} we will also use the following notations. We consider a compact space $B$ equipped with a nowhere vanishing Radon measure $db$. Let us say that a function $\phi : B \times \bbR_{\ge 0}^n \to \bbC$ is \textbf{nice} if it is expressible as $$ B \times \bbR_{\ge 0}^n \xrightarrow{{\rm id}_B \times {\rm ei}} B \times D^n \xrightarrow{\phi^{\circ}} \bbC$$ where ${\rm ei}(x_1 , \ldots , x_n) := (e^{-x_1} , \ldots , e^{-x_n})$ and $\phi^{\circ}$ is continuous and holomorphic in the second variable (in the sense that when we fix the variable in $B$ it is the restriction of a holomorphic function on a neighbourhood of $D^n$). Given $J \subset [n]$ we denote by ${\rm res}_{J} \phi : B \times \bbR_{\ge 0}^J \to \bbC$ the function given by ${\rm res}_{J} \phi (b , y) := \phi^{\circ} (b , {\rm ext}^{J} ({\rm ei} (y)))$. We also write $\phi (b , +\infty)$ for $\phi^{\circ} (b , 0)$ etc.
\subsection{Growth - the case of summation over a lattice}

\begin{lemma}\label{lem appb lat 1}
	Let $\lambda := (\lambda_1 , \ldots , \lambda_n) \in \bbC_{\leq 0}^n$ and $m := (m_1 , \ldots , m_n) \in \bbZ_{\ge 0}$. Let $K \subset \bbR_{\ge 0}^n$ be a compact subset. Assume that ${\rm Re} (\lambda) = 0$ and $\lambda \notin (2\pi i) \bbZ^n$. We have $$ \sup_{Q \subset K} \left| \frac{1}{r^n} \sum_{x \in \frac{1}{r} \bbZ_{\ge 0}^n \cap Q} x^m e^{r \langle \lambda , x \rangle} \right| = O( r^{-1} )$$ as $r \to +\infty$, where $Q$ denote convex subsets.
\end{lemma}

\begin{proof}
	Let us re-order the variables, assuming that $\lambda_1 \notin 2\pi i \bbZ$. Let us write $x = (x_1 , x')$ where $x' = (x_2 , \ldots , x_n)$ and analogously write $m'$ et cetera. Given a convex subset $Q \subset K$ and $x' \in \bbR_{\ge 0}^{n-1}$ let us denote by $Q^{x'} \subset \bbR_{\ge 0}$ the subset consisting of $x_1$ for which $(x_1 , x') \in Q$ (it is an interval). Let us enlarge $K$ for convenience, writing it in the form $K = K_1 \times K'$ where $K_1 \subset \bbR_{\ge 0}$ is a closed interval and $K' \subset \bbR_{\ge 0}^{n-1}$ is the product of closed intervals.
	
	\medskip
	
	We have $$ \sum_{x \in \frac{1}{r} \bbZ_{\ge 0}^n \cap Q} x^m e^{r\langle \lambda , x \rangle} = \sum_{ x' \in \frac{1}{r} \bbZ_{\ge 0}^{n-1} \cap K' } (x')^{m'} e^{r \langle \lambda' , x' \rangle }\left( \sum_{ x_1 \in \frac{1}{r} \bbZ_{\ge 0} \cap Q^{x'}} x_1^{m_1} e^{r \lambda_1 x_1} \right).$$ We have $Q^{x'} \subset K'$ and it is elementary to see that $$ \sup_{R \subset K'} \left| \sum_{x_1 \in \frac{1}{r} \bbZ_{\ge 0} \cap R} x_1^{m_1} e^{r \lambda_1 x_1}  \right| = O(1)$$ as $r \to +\infty$, where $R$ denote intervals. Therefore we obtain, for some $C>0$ (not depending on $Q$) and all $r \ge 1$: $$ \left| \frac{1}{r^n} \sum_{x \in \frac{1}{r} \bbZ_{\ge 0}^n \cap Q} x^m e^{r\langle \lambda , x \rangle} \right| \leq C \left(  \frac{1}{r^{n-1}} \sum_{ x' \in \frac{1}{r} \bbZ_{\ge 0}^{n-1} \cap K' } (x')^{m'} \right) r^{-1}.$$ Since the expression in brackets is clearly bounded independently of $r$, we are done.
\end{proof}

\begin{lemma}\label{lem appb lat 2} Let $(\lambda , m ) \in \bbC_{\leq 0}^n \times \bbZ_{\ge 0}^n$. Then the limit $$ \lim_{r \to +\infty} \frac{1}{r^{d(\lambda , m)}} \sum_{x \in \bbZ_{\ge 0}^n \cap P_{<r}} x^m e^{\langle \lambda , x \rangle} dx$$ exists, equal to $0$ if ${\rm res}_{J_{\lambda}} (\lambda) \notin 2\pi i \cdot \bbZ^{J_{\lambda}}$ and otherwise equal to $$\left( \int_{P_{J_{\lambda}}} y^{{\rm res}_{J_{\lambda}} (m)} dy \right) \left( \sum_{z \in \bbZ_{\ge 0}^{J_{\lambda}^c}} z^{{\rm res}_{J_{\lambda}^c} (m)} e^{\langle {\rm res}_{J_{\lambda}^c} (\lambda) , z \rangle } \right)$$ (the sum converging absolutely).
\end{lemma}

\begin{proof}
	Let us abbreviate $J := J_{\lambda}$. Let us denote $\lambda' := {\rm res}_J (\lambda)$ and $\lambda'' := {\rm res}_{J^c} (\lambda)$, and similarly for $m$. Given $x'' \in \bbZ_{\ge 0}^{J_c}$ let us denote by $P_{(<r)}^{x''} \subset \bbR_{\ge 0}^{J}$ the subset consisting of $y'$ for which ${\rm ext}^{J} (r y') + {\rm ext}^{J^c} (x'') \in P_{<r}$.
	
	\medskip
	
	We have $$ \sum_{x \in \bbZ_{\ge 0}^n \cap P_{<r}} x^m e^{\langle \lambda , x \rangle} = r^{d(\lambda , m) - |J|} \sum_{x'' \in \bbZ_{\ge 0}^{J^c}} (x'')^{m''} e^{\langle \lambda'' , x'' \rangle} \sum_{y' \in \frac{1}{r} \bbZ_{\ge 0}^{J} \cap P^{x''}_{(<r)}} (y')^{m'} e^{r \langle \lambda' , y' \rangle} := \triangle.$$
	
	\medskip
	
	Let us assume first that $\lambda' \notin 2 \pi i \cdot \bbZ^{J}$. Then by Lemma \ref{lem appb lat 1} there exists $C>0$ such that for all convex subsets $Q \subset P_J$ and all $r \ge 1$ we have $$ \left| \frac{1}{r^{|J|}} \sum_{y' \in \frac{1}{r} \bbZ_{\ge 0}^{J} \cap Q} (y')^{m'} e^{r \langle \lambda' , y' \rangle} \right| \leq C \cdot r^{-1}.$$ Therefore $$ | \triangle | \leq C r^{d(\lambda , m) - 1}  \sum_{x'' \in \bbZ_{\ge 0}^{J^c}} (x'')^{m''} e^{\langle {\rm Re} (\lambda'') , x'' \rangle}, $$ giving the desired.
	
	\medskip
	
	Now we assume $\lambda' \in 2 \pi i \cdot \bbZ^J$. It is not hard to see that $$ \lim_{r \to +\infty} \frac{1}{r^{|J|}} \sum_{y' \in \frac{1}{r} \bbZ_{\ge 0}^{J} \cap P^{x''}_{(<r)}} (y')^{m'} = \int_{P_J} (y')^{m'} dy'.$$ Hence we have (by dominated convergence) $$ \lim_{r \to +\infty} \frac{1}{r^{d(\lambda,m)}} \triangle =  \sum_{x'' \in \bbZ_{\ge 0}^{J^c}} (x'')^{m''} e^{\langle \lambda'' , x'' \rangle} \int_{P_J} (y')^{m'} dy' .$$
\end{proof}

\begin{claim}\label{clm appb lat}
	Let $p \ge 1$, let $\{ ( \lambda^{(\ell)} , m^{(\ell)}) \}_{\ell \in [p]} \subset \bbC_{\leq 0}^{n} \times \bbZ_{\ge 0}^n$ be a collection of pairwise different couples and let $\{ c^{(\ell)} \}_{\ell \in [p]} \subset \bbC \smallsetminus \{ 0 \}$ be a collection of non-zero scalars. Denote $d := \max_{\ell \in [p]} d(2{\rm Re} (\lambda^{(\ell)}) , 2 m^{(\ell)})$.
	The limit $$ \lim_{r \to +\infty} \frac{1}{r^d} \sum_{x \in \bbZ_{\ge 0}^n \cap P_{<r}} \left| \sum_{\ell \in [p]} c^{(\ell)} x^{m^{(\ell)}} e^{\langle \lambda^{(\ell)} , x \rangle } \right|^2 $$ exists and is strictly positive.
\end{claim}

\begin{proof}
	Let us break the integrand into a sum following $$\left| \sum_{\ell \in [p]} A_{\ell} \right|^2 = \sum_{\ell_1 , \ell_2 \in [p]} A_{\ell_1} \overline{A_{\ell_2}}.$$ Using Lemma \ref{lem appb lat 2} we see the that resulting limit breaks down as a sum, over $(\ell_1 ,\ell_2) \in [p]^2$, of limits which exist, so the only thing to check is that the resulting limit is non-zero. It is easily seen that the limit at the $(\ell_1 , \ell_2)$ place is zero unless $d(\lambda^{(\ell_1)} , m^{(\ell_1)} ) = d$, $d(\lambda^{(\ell_2)} , m^{(\ell_2)} ) = d$, $J_{\lambda^{(\ell_1)}} = J_{\lambda^{(\ell_2)}}$ and ${\rm res}_{J^{(\ell_1)}} (\lambda^{(\ell_2)}) - {\rm res}_{J^{(\ell_1)}} (\lambda^{(\ell_1)}) \in 2\pi i \cdot \bbZ^J$. We thus can reduce to the case when, for a given $J \subset [n]$, we have $J_{\lambda^{(\ell)}} = J$ for all $\ell \in [p]$, we have $d(\lambda^{(\ell)} , m^{(\ell)}) = d$ for all $\ell \in [p]$, and we have ${\rm res}_{J} (\lambda^{(\ell_2)}) - {\rm res}_{J} (\lambda^{(\ell_1)}) \in 2\pi i \cdot \bbZ^J$ for all $\ell_1 , \ell_2 \in [p]$. We then obtain, using Lemma \ref{lem appb lat 2}, that our overall limit equals $$ \sum_{z \in \bbZ_{\ge 0}^{J^c}} \int_{P_J} \left| \sum_{\ell \in [p]} c^{(\ell)} y^{{\rm res}_J (m^{(\ell)})} z^{{\rm res}_{J^c} (m^{(\ell)})} e^{\langle {\rm res}_{J^c} (\lambda^{(\ell)}) , z \rangle } \right|^2 dy.$$ It is therefore enough to check that $$   \sum_{\ell \in [p]} c^{(\ell)} y^{{\rm res}_J (m^{(\ell)})} z^{{\rm res}_{J^c} (m^{(\ell)})} e^{\langle {\rm res}_{J^c} (\lambda^{(\ell)}) , z \rangle }, $$ a function in $(z,y ) \in  \bbZ_{\ge 0}^{J^c} \times P_J$, is not identically zero. By the local linear independence of powers of $y$, we can further assume that ${\rm res}_J (m^{(\ell)})$ is independent of $\ell \in [p]$, and want to check that $$\sum_{\ell \in [p]} c^{(\ell)} z^{{\rm res}_{J^c} (m^{(\ell)})} e^{\langle {\rm res}_{J^c} (\lambda^{(\ell)}) , z \rangle }, $$ a function in $z \in  \bbZ_{\ge 0}^{J^c}$, is not identically zero. Notice that, by our assumptions, the elements in the collection $ \{ ( {\rm res}_{J^c} (\lambda^{(\ell)})  , {\rm res}_{J^c} (m^{(\ell)})  ) \}_{\ell \in [p]}$ are pairwise different. Thus the non-vanishing of our sum is clear (by linear algebra of generalized eigenvectors of shift operators on $\bbZ^{J^c}$).
\end{proof}

\subsection{Growth - the case of an integral}\label{ssec growth integral}

\begin{lemma}\label{lem appb 1}
	Let $\lambda := (\lambda_1 , \ldots , \lambda_n) \in \bbC_{\leq 0}^n$ and $m := (m_1 , \ldots , m_n) \in \bbZ_{\ge 0}$. Let $K \subset \bbR_{\ge 0}^n$ be a compact subset. Assume that ${\rm Re} (\lambda) = 0$ and $\lambda \neq 0$. We have $$ \sup_{Q \subset K} \left| \int_Q x^m e^{r \langle \lambda , x \rangle} dx \right| = O( r^{-1} )$$ as $r \to +\infty$, where $Q$ denote convex subsets.
\end{lemma}

\begin{proof}
	Let us re-order the variables, assuming that $\lambda_1 \neq 0$. Let us write $x = (x_1 , x')$ where $x' = (x_2 , \ldots , x_n)$ and analogously write $m'$ etcetera. Given a convex subset $Q \subset K$ and $x' \in \bbR_{\ge 0}^{n-1}$ let us denote by $Q^{x'} \subset \bbR_{\ge 0}$ the subset consisting of $x_1$ for which $(x_1 , x') \in Q$ (it is an interval). Let us enlarge $K$ for convenience, writing it in the form $K = K_1 \times K'$ where $K_1 \subset \bbR_{\ge 0}$ is a closed interval and $K' \subset \bbR_{\ge 0}^{n-1}$ is the product of closed intervals.
	
	\medskip
	
	Using Fubini's theorem $$ \int_Q x^m e^{r\langle \lambda , x \rangle} dx = \int_{ K' } (x')^{m'} e^{r \langle \lambda' , x' \rangle }\left( \int_{Q^{x'}} x_1^{m_1} e^{r \lambda_1 x_1} dx_1 \right) dx'.$$ We have $Q^{x'} \subset K'$ and it is elementary to see that $$ \sup_{R \subset K'} \left| \int_{R} x_1^{m_1} e^{r \lambda_1 x_1} dx_1 \right| = O(r^{-1})$$ as $r \to +\infty$, where $R$ denote intervals. Therefore we obtain, for some $C>0$ and all $r \ge 1$: $$ \left|  \int_Q x^m e^{r\langle \lambda , x \rangle} dx  \right| \leq C \left( \int_{K'} (x')^{m'} dx' \right) r^{-1},$$ as desired.
\end{proof}

\begin{lemma}\label{lem appb 2} Let $(\lambda , m ) \in \bbC_{\leq 0}^n \times \bbZ_{\ge 0}^n$ and let $\phi : B \times \bbR_{\ge 0}^n \to \bbC$ be a nice function. Then the limit $$ \lim_{r \to +\infty} \frac{1}{r^{d(\lambda , m)}} \int_B \int_{P_{<r}} x^m e^{\langle \lambda , x \rangle}  \phi (b,x)dxdb$$ exists, equal to $0$ if ${\rm res}_{J_{\lambda}} \lambda \neq 0$ and otherwise equal to $$\left( \int_{P_{J_{\lambda}}} y^{{\rm res}_{J_{\lambda}} (m)} dy \right) \left( \int_B \int_{\bbR_{\ge 0}^{J_{\lambda}^c}} z^{{\rm res}_{J_{\lambda}^c} (m)} e^{\langle {\rm res}_{J_{\lambda}^c} (\lambda) , z\rangle} {\rm res}_{J_{\lambda}^c} \phi (b , z) dz db \right)$$ (the double integral converging absolutely).
\end{lemma}

\begin{proof}

Let us re-order the variables, assuming that $J := J_{\lambda} = [k]$. Let us write $x = (x',x'')$ where $x'$ consists of the first $k$ components and $x''$ consists of the last $k$ components. Let us write analogously $m' , \lambda'$ etc.

	\medskip
	
	First, let us notice that if $k \neq 0$, we can write $$ \phi (b , x ) = e^{-x_1} \phi_0 (b , x) + \phi_1 (b , x)$$ where $\phi_0 , \phi_1 : B \times \bbR_{\ge 0}^n \to \bbC$ are nice functions and $\phi_1$ does not depend on $x_1$. Dealing with $ e^{-x_1} \phi_0 (b,x)$ instead of $\phi (b,x)$ makes us consider $\lambda$ with smaller set $J_{\lambda}$ and thus $(\lambda , m)$ with a smaller $d(\lambda , m)$ and from this, reasoning inductively, we see that we can assume that $\phi$ only depends on $(b,x'')$. Let us write $\phi'' := {\rm res}_{J^c} \phi$.
	
	\medskip
	
	Let us perform a change of variables $y' := \frac{1}{r} x'$. Let $P_{(<r)} \subset \bbR_{\ge 0}^n$ denote the transform of $P_{<r}$ under this changes of variables (i.e. $(x',x'') \in P_{<r}$ if and only if $(y',x'') \in P_{(<r)}$). We obtain $$ \int_B \int_{P_{<r}} x^m e^{\langle \lambda , x \rangle} \phi (b , x) dx db = $$ $$ = r^{d}  \int_B \int_{P_{(<r)}} (y')^{m'} e^{r \langle \lambda' , y' \rangle } (x'')^{m''} e^{\langle \lambda'' , x'' \rangle} \phi'' (b , x'') dy' dx'' db =: \triangle.$$
	
	Given $x'' \in \bbR_{\ge 0}^{J^c}$, let us denote by $P_{(<r)}^{x''} \subset \bbR_{\ge 0}^{J}$ the set consisting of $y'$ for which $(y',x'') \in P_{(<r)}$. Notice that $P_{(<r_1)}^{x''} \subset P_{(<r_2)}^{x''}$ for $r_1 < r_2$ and $\cup_{r} P_{(<r)}^{x''} = P_{J}$. Using Fubini's theorem
	
	$$ \triangle = r^{d} \int_B \int_{ \bbR_{\ge 0}^{J^c}} (x'')^{m''} e^{\langle \lambda'' , x'' \rangle } \phi'' (b , x'') \left( \int_{ P_{(<r)}^{x''}} (y')^{m'} e^{r \langle \lambda' , y' \rangle }  dy'  \right) dx'' db.$$
	
	\medskip
	
	If $\lambda' \neq 0$, by Lemma \ref{lem appb 1} there exists $C>0$ such that for all convex subsets $Q \subset P_J$ and all $r \ge 1$ we have $$ \left| \int_Q (y')^{m'} e^{r \langle \lambda' , y' \rangle } dy' \right| \leq C \cdot r^{-1}.$$ We have therefore $$ \left| \triangle \right| \leq C \cdot r^{d - 1} \cdot  \int_B \int_{ \bbR_{\ge 0}^{J^c}} (x'')^{m''} e^{\langle {\rm Re} (\lambda'' ) , x'' \rangle } | \phi'' (b , x'')| dx'' db $$ and thus indeed the desired limit is equal to $0$.
	
	\medskip
	
	Now we assume $\lambda' = 0$. Using Lebesgue's dominated convergence theorem we have $$ \lim_{r \to +\infty} \frac{1}{r^d} \triangle =  \lim_{r \to +\infty} \int_B \int_{\bbR_{\ge 0}^{J^c}} (x'')^{m''} e^{\langle \lambda'' , x'' \rangle } \phi'' (b , x'') \left( \int_{ P_{(<r)}^{x''}} (y')^{m'}  dy'  \right) dx'' db = $$ $$ = \int_B \int_{ \bbR_{\ge 0}^{J^c}} (x'')^{m''} e^{\langle \lambda'' , x'' \rangle } \phi'' (b , x'') \left( \int_{P_J} (y')^{m'}  dy'  \right) dx'' db$$ as desired.
\end{proof}

\begin{claim}\label{clm appb}
	Let $\{ ( \lambda^{(\ell)} , m^{(\ell)}) \}_{\ell \in [p]} \subset \bbC_{\leq 0}^{n} \times \bbZ_{\ge 0}^n$ be a collection of pairwise different couples. Let $\{ \phi^{(\ell)} \}_{\ell \in [p]}$ be a collection of nice functions $B \times \bbR_{\ge 0}^n \to \bbC$, such that for every $\ell \in [p]$ the function $b \mapsto \phi^{(\ell)} (b , +\infty)$ on $B$ is not identically zero. Denote $d := \max_{\ell \in [p]} d(2{\rm Re} (\lambda^{(\ell)}) , 2 m^{(\ell)})$.
	The limit $$ \lim_{r \to +\infty} \frac{1}{r^d} \int_B \int_{P_{<r}} \left| \sum_{\ell \in [p]} x^{m^{(\ell)}} e^{\langle \lambda^{(\ell)} , x \rangle } \phi^{(\ell)} (b , x)\right|^2 dx db$$ exists and is strictly positive.
\end{claim}

\begin{proof}
	Let us break the integrand into a sum following $$\left| \sum_{\ell \in [p]} A_{\ell} \right|^2 = \sum_{\ell_1 , \ell_2 \in [p]} A_{\ell_1} \overline{A_{\ell_2}}.$$ Using Lemma \ref{lem appb 2} we see the that resulting limit breaks down as a sum, over $(\ell_1 ,\ell_2) \in [p]^2$, of limits which exist, so the only thing to check is that the resulting limit is non-zero. It is easily seen that the limit at the $(\ell_1 , \ell_2)$ place is zero unless $d(\lambda^{(\ell_1)} , m^{(\ell_1)} ) = d$, $d(\lambda^{(\ell_2)} , m^{(\ell_2)} ) = d$, $J_{\lambda^{(\ell_1)}} = J_{\lambda^{(\ell_2)}}$ and ${\rm res}_{J^{(\ell_1)}} (\lambda^{(\ell_1)}) = {\rm res}_{J^{(\ell_1)}} (\lambda^{(\ell_2)})$. We thus can reduce to the case when, for a given $J \subset [n]$, we have $J_{\lambda^{(\ell)}} = J$ for all $\ell \in [p]$, we have $d(\lambda^{(\ell)} , m^{(\ell)}) = d$ for all $\ell \in [p]$, and we have ${\rm res}_{J} (\lambda^{(\ell_1)}) = {\rm res}_{J} (\lambda^{(\ell_2)})$ for all $\ell_1 , \ell_2 \in [p]$. We then obtain, using Lemma \ref{lem appb 2}, that our overall limit equals $$ \int_B \int_{\bbR_{\ge 0}^{J^c}} \int_{P_J} \left|  \sum_{\ell \in [p]} y^{{\rm res}_J (m^{(\ell)})} z^{{\rm res}_{J^c} (m^{(\ell)})} e^{\langle {\rm res}_{J^c} (\lambda^{(\ell)}) , z\rangle} {\rm res}_{J^c} \phi^{(\ell)} (b , z) \right|^2 dy dz db.$$ It is therefore enough to check that $$  \sum_{\ell \in [p]} y^{{\rm res}_J (m^{(\ell)})} z^{{\rm res}_{J^c} (m^{(\ell)})} e^{\langle {\rm res}_{J^c} (\lambda^{(\ell)}) , z\rangle} {\rm res}_{J^c} \phi^{(\ell)} (b , z),$$ a function in $(b,z,y ) \in B \times \bbR_{\ge 0}^{J^c} \times P_J$, is not identically zero. By the local linear independence of powers of $y$, we can further assume that ${\rm res}_J (m^{(\ell)})$ is independent of $\ell \in [p]$, and want to check that $$ \sum_{\ell \in [p]} z^{{\rm res}_{J^c} (m^{(\ell)})} e^{\langle {\rm res}_{J^c} (\lambda^{(\ell)}) , z\rangle} {\rm res}_{J^c} \phi^{(\ell)} (b , z),$$ a function in $(b,z) \in B \times \bbR_{\ge 0}^{J^c}$, is not identically zero. Notice that, by our assumptions, the elements in the collection $ \{ ( {\rm res}_{J^c} (\lambda^{(\ell)})  , {\rm res}_{J^c} (m^{(\ell)})  ) \}_{\ell \in [p]}$ are pairwise different and for every $\ell \in [p]$, the function $b \mapsto \phi^{(\ell)} (b , {\rm ext}^{J^c} (+\infty))$ on $B$ is not identically zero. Considering the partial order on $\bbC^{J^c}$ given by $\mu_1 \leq \mu_2$ if $\mu_2 - \mu_1 \in \bbZ_{\ge 0}^{J^c}$, we can pick $\ell \in [p]$ for which ${\rm res}_{J^c} (\lambda^{(\ell)})$ is maximal among the $\{ {\rm res}_{J^c} (\lambda^{(\ell')}) \}_{\ell' \in [p]}$. We can then pick $b \in B$ such that $\phi^{(\ell)} (b , {\rm ext}^{J^c} (+\infty)) \neq 0$. We then boil down to Lemma \ref{lem appb 4} that follows.
	
\end{proof}

In the end of the proof of Claim \ref{clm appb} we have used the following:

\begin{lemma}\label{lem appb 4}
	Let $\{ (\lambda^{(\ell)} , m^{(\ell)}) \}_{\ell \in [p]} \subset \bbC^n \times \bbZ_{\ge 0}^n$ be a collection of pairwise different couples. Let $\{ \phi^{(\ell)}\}_{\ell \in [p]}$ be a collection of nice functions $\bbR_{\ge 0}^n \to \bbC$ (so here $B = \{ 1 \}$). Suppose that $\phi^{(\ell)} (+\infty) \neq 0$ for some $\ell \in [p]$ for which $\lambda^{(\ell)}$ is maximal among the $\{ \lambda^{(\ell')} \}_{\ell' \in [p]}$ with respect to the partial order $\lambda_1 \leq \lambda_2$ if $\lambda_2 - \lambda_1 \in \bbZ_{\ge 0}^n$. Then the function $$ x \mapsto \sum_{\ell \in [p]} x^{m^{(\ell)}} e^{\langle \lambda^{(\ell)} , x \rangle} \phi^{(\ell)} (x)$$ on $\bbR_{\ge 0}^n$ is not identically zero.
\end{lemma}

\begin{proof}
	We omit the proof - one develops the $\phi^{(\ell)}$ into power series in $e^{-x_1} , \ldots , e^{-x_n}$ and uses separation by generalized eigenvalues of the partial differentiation operators $\partial_{x_1} , \ldots , \partial_{x_n}$.
\end{proof}

\end{document}